\documentclass[12pt]{amsart}
\usepackage{amsmath, amssymb, amsthm}

\usepackage{geometry}
\usepackage{graphicx}
\usepackage{tikz}
\newtheorem{thm}{Theorem}[section]

\newtheorem{prop}[thm]{Proposition}

\newtheorem{lem}[thm]{Lemma}
\theoremstyle{definition}
\newtheorem{ex}{Example}  
\newtheorem{defn}[thm]{Definition}

\newtheorem{remark}{Remark}
\usepackage{tikz}
\usepackage{placeins}
\usepackage{subcaption}
\usepackage{comment}

\usepackage[colorlinks=true, linkcolor=blue]{hyperref}
\usetikzlibrary{arrows.meta}
\usetikzlibrary{calc}

\title{Asymptotic face distributions in random reduced $\mathfrak s\mathfrak l_3$ webs}
\author{David Kogan}
\address{Department of Mathematics, Yale University, New Haven, CT 06511, USA}
\email{d.kogan@yale.edu}
\date{October 1, 2025}

\begin{document}

\begin{abstract}
We study the distribution of interior faces in uniformly random reduced $ \mathfrak{sl}_3 $ webs.  
Using Tymoczko’s bijection between $ 3 \times n $ standard Young tableaux and reduced webs, this problem can be reformulated in terms of constrained lattice paths and associated $m$-diagrams.  
We develop a framework that expresses crossing probabilities in the $m$-diagram as solutions to discrete Dirichlet problems on the triangular lattice, which are evaluated through solutions to lattice Green’s functions.  
From this we obtain explicit limiting formulas for the frequencies of interior faces of each type.  

As an application, we analyze faces at a distance at least $d$ from the boundary.  We prove that almost all interior faces far from the boundary are hexagons, while faces of size $6+2k$ occur with probability $O(d^{-2k})$. 
\end{abstract}

\maketitle

\section{Introduction}

Reduced $ \mathfrak{sl}_3 $ webs are planar trivalent graphs embedded in a disk whose interior faces have size at least six. They were introduced by Kuperberg \cite{kuperberg1996spiders} as combinatorial bases for invariant functions of tensor products of $ \mathfrak{sl}_3 $ representations. Khovanov and Kuperberg \cite{khovanov1999web} developed a 
recursive algorithm that gives a bijection between reduced webs and $3\times n$ tableaux.

Extending this result, Tymoczko~\cite{tymoczko2007webs} found an explicit map from standard Young tableaux of shape $3 \times n$ to reduced webs with $3n$ boundary vertices. This correspondence allows us to study uniformly random webs via uniformly random tableaux, or equivalently via constrained lattice paths in the nonnegative quadrant. The scaling limit of these paths to a Brownian excursion in two dimensions  provides a tool to study the local statistics of faces and crossings in random webs.

Our approach is to track how crossings of arcs in Tymoczko’s $m$--diagram representation give rise to faces. 
We develop a framework that encodes the probabilities of crossing configurations (patterns of arc intersections) in terms of discrete harmonic functions on the triangular lattice and associated Dirichlet problems in wedge domains. 
In this correspondence, subpaths of the constrained lattice walk represent local portions of the $m$--diagram and determine the structure of individual faces. 
This yields explicit limiting probabilities for faces of a given type, determined by their sequence of crossings. For instance, we compute that the probability a newly opened first arc of an $m$ closes without undergoing any further crossings is 
$\frac{243\sqrt 3}{40\pi}-3$. More generally, such crossing probabilities are in principle computable, with exact evaluations available in some cases and contour integral formulas in others.

The proofs proceed in several stages. 
In Section~\ref{sec:interior_faces}, we recall Tymoczko’s bijection and formalize how interior faces correspond to alternating red/blue arc sequences in an $m$--diagram. 
In Section~\ref{sec:crossing_probabilities}, we analyze crossing events by encoding subpaths of the constrained lattice walk as axis–hitting problems. 
This reduces the probability of any prescribed crossing configuration to the evaluation of certain discrete harmonic functions $h_a$ and $g$, which solve Dirichlet boundary problems in wedge domains. 
In Section~\ref{sec:green_function}, we evaluate these functions via lattice Green’s functions, and obtain integral formulas. 
This yields the general limiting distribution of faces of each type (Theorem~\ref{thm:face_probabilities}). 
Finally, in Section~\ref{sec:faces_m}, we extend the analysis to faces lying at depth $d$ from the boundary. 
Here, we introduce the notion of face diagram extension and prove that each extension lowers the probability of a larger face by a factor of order $d^{-2}$, implying that almost all bulk faces are hexagons. 

Our main theorem states that the bulk of a large random reduced $\mathfrak{sl}_3$ web is asymptotically hexagonal:

\begin{thm}\label{thm:face_asymptotics}
Let $d \to \infty$ with $n \to \infty$ and $d = o(n)$.  
For faces at distance at least $d$ from the boundary in a uniformly random reduced $\mathfrak s\mathfrak l_3$ web, we have:
\begin{enumerate}
    \item The probability that a uniformly chosen such face has size $6$ converges to $1$ as $d \to \infty$.
    \item More generally, for finite $d$, the probability that such a face has size $6 + 2k$ is $O(d^{-2k})$ for $k \ge 1$, where the constant depends only on $k$.
\end{enumerate}
\end{thm}

\begin{figure}[htbp]
    \centering
    \includegraphics[width=0.6\textwidth]{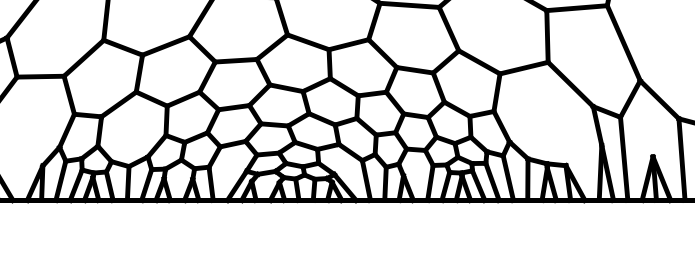}
    \caption{Zoomed $O(1)$ window of a reduced $\mathfrak{sl}_3$ web in the upper half–plane. 
The boundary lies on the real line. Moving rightward, i.i.d. boundary moves open and close strands, 
producing the local web configuration.}
\label{fig:partial_web}
\end{figure}

Figure \ref{fig:random_web} shows a simulation of random reduced $\mathfrak s\mathfrak l_3$ web.

\begin{figure}[htbp]
    \centering
\includegraphics[width=0.7\textwidth]{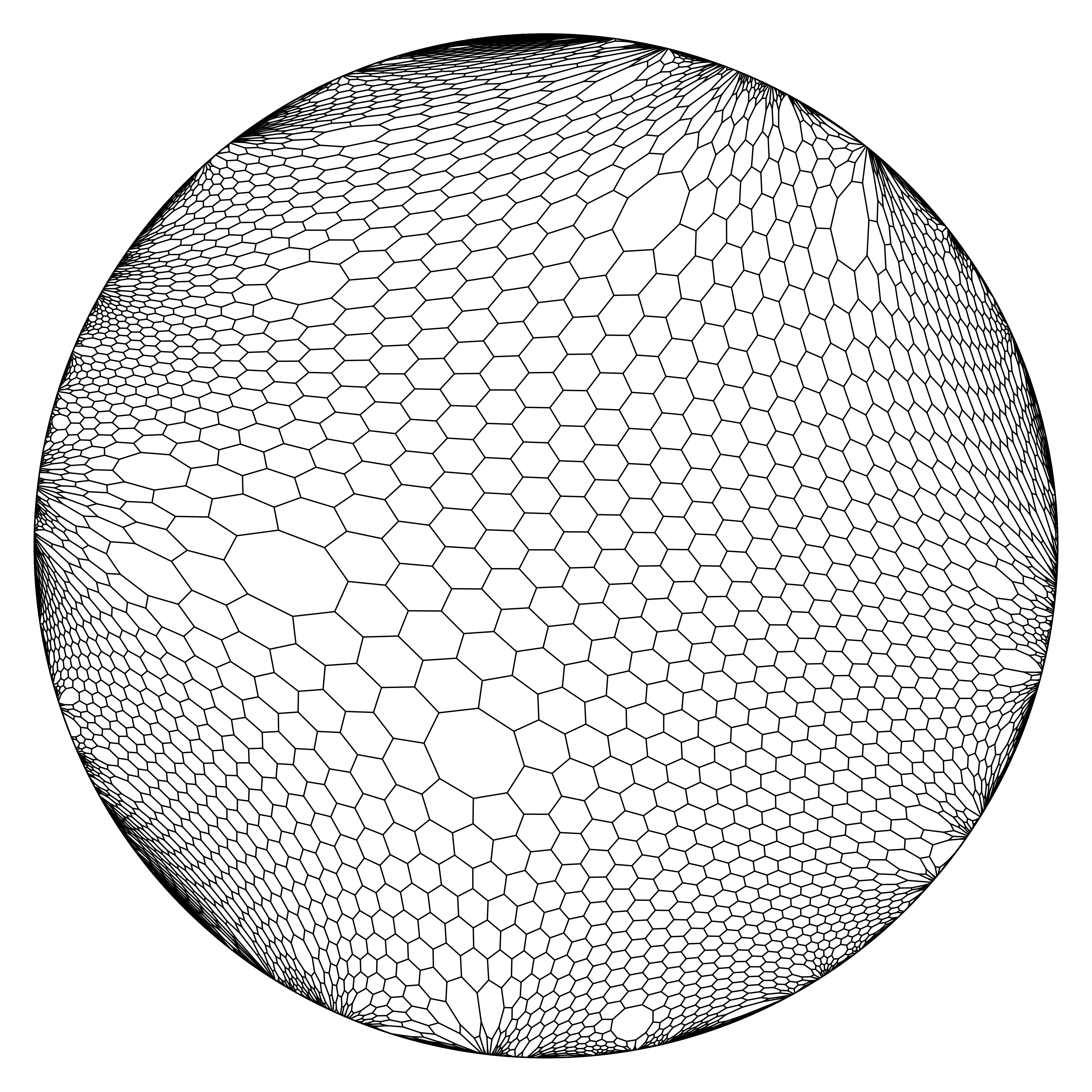}
    \caption{Uniform random reduced web with $n=800$ ($2400$ boundary vertices)}
    \label{fig:random_web}
\end{figure}

\section*{Acknowledgements}
The author is grateful to Richard Kenyon for suggesting the problem and for helpful discussions and insights.

\section{Interior Faces of \texorpdfstring{$\mathfrak s \mathfrak l_3$}{sl3} webs and \texorpdfstring{$m$}{m}-diagrams}
\label{sec:interior_faces}

To analyze random reduced webs we first recall Tymoczko’s bijection between 
$3 \times n$ tableaux and $m$--diagrams. This representation encodes each web 
as a collection of arcs above a line, so that every interior face of the web 
corresponds to a region bounded by arcs in the $m$--diagram. In this section we 
review the bijection and describe how interior faces arise in this picture.

\FloatBarrier

\subsection{Tymoczko's Bijection} 

\label{sec:Tymoczko's Bijection}

\begin{defn}
A \emph{web} for $\mathfrak{sl}_3$ is a planar bipartite directed graph embedded in a disk with the following properties:
\begin{enumerate}
    \item Each internal vertex is trivalent and has all incident edges either directed inwards (a sink) or outwards (a source).
    \item Each boundary vertex has degree one and all boundary edges are oriented outwards.
\end{enumerate}
A web is \emph{reduced} (or \emph{non–elliptic}) if every interior face is bounded by at least six edges.  
Let $W_n$ denote the set of reduced $\mathfrak{sl}_3$ webs with $3n$ boundary vertices.
\end{defn}

 In this section we review Tymoczko's bijection between $3\times n$ standard Young tableaux and reduced $\mathfrak s\mathfrak l_3$ reduced webs of size $n$ from \cite{tymoczko2007webs}. Let $T$ be a standard Young tableaux of shape $3\times n$ with entries $\{1,2,\dots,3n\}$. The corresponding reduced $\mathfrak{sl}_3$ web $W=\Phi(T)$ is obtained in two steps. First, construct the $m$-diagram of $T$, and then produce a reduced web from the $m$-diagram.

\begin{defn}[SYT to $m$-diagrams]
We start from a $3\times n$ standard Young tableaux $T$.
\begin{enumerate}
    \item Draw a horizontal line with $3n$ marked points labeled $1,2,\dots,3n$ from left to right. This line is the boundary of the $m$-diagram, and all arcs lie above it.
    \item For each $i=1, \dots, 3n$ not on the bottom row, find $j<i$ 
    \begin{itemize}
        \item such that $j$ lies in the row immediately below the row $i$.
        \item $j$ is the largest number that is not already on an arc with another number from the same row as $i$. 
        \item Join $i$ to $j$ with a semicircular arc.
    \end{itemize}
\end{enumerate}
An \emph{$m$} in the diagram consists of two arcs $(i,j)$ and $(j,k)$ with $i<j<k$, where $(i,j)$ is a first arc and $(j,k)$ is a second arc. The triple $i, j, k$ are in the first, second, and third row of the tableaux $T$, respectively. We will use the color red to refer to first arcs and blue to refer to second arcs. See figure \ref{fig:3x3-tableau-mdiagram}.
\end{defn}

\begin{figure}[ht]
\centering
\begin{tikzpicture}[scale=0.95, every node/.style={font=\small}]

  \begin{scope}[shift={(0,0)}]
    \node[anchor=west,font=\bfseries] at (-0.2,3.7) {(a) Tableaux};

    \foreach \x in {0,1,2} {
      \foreach \y in {0,1,2} {
        \draw (\x,\y) rectangle ++(1,1);
      }
    }

    \node at (0.5,0.5) {1};
    \node at (1.5,0.5) {2};
    \node at (2.5,0.5) {3};
    \node at (0.5,1.5) {4};
    \node at (1.5,1.5) {5};
    \node at (2.5,1.5) {7};
    \node at (0.5,2.5) {6};
    \node at (1.5,2.5) {8};
    \node at (2.5,2.5) {9};
  \end{scope}

  \begin{scope}[shift={(6.2,0)}]
    \node[anchor=west,font=\bfseries] at (-0.2,3.7) {(b) $m$-diagram};

    \draw (0.8,0) -- (9.2,0);
    \foreach \i in {1,...,9} {
      \draw (\i,0) -- (\i,0.08);
      \node[below] at (\i,0) {\i};
    }

    \begin{scope}[very thick,red]
      \draw (3,0) to[out=90,in=90] (4,0);
      \draw (2,0) to[out=90,in=90] (5,0);
      \draw (1,0) to[out=90,in=90] (7,0);
    \end{scope}

    \begin{scope}[very thick,blue]
      \draw (7,0) to[out=90,in=90] (8,0);
      \draw (5,0) to[out=90,in=90] (6,0);
      \draw (4,0) to[out=90,in=90] (9,0);
    \end{scope}

    \draw[red,very thick] (0.4,3.2)--(1.2,3.2) node[right,black] {first arcs (row 2 $\to$ row 1)};
    \draw[blue,very thick]  (0.4,2.8)--(1.2,2.8) node[right,black] {second arcs (row 3 $\to$ row 2)};
  \end{scope}
\end{tikzpicture}
\caption{A $3\times3$ standard Young tableaux (a) and its $m$-diagram (b). Red arcs join second-row entries to first-row entries; blue arcs join third-row entries to second-row entries.}
\label{fig:3x3-tableau-mdiagram}
\end{figure}
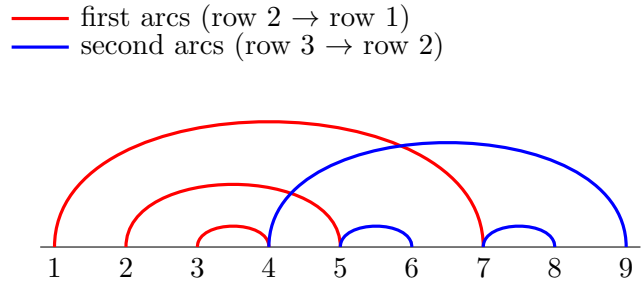

\FloatBarrier

The following lemma was proven in \cite{tymoczko2007webs}.
\begin{lem} \label{lem:ms_conditions}
    In any $m$-diagram the following hold:
    \begin{itemize}
        \item at most two arcs cross at any point
        \item if two arcs cross they must be of different color
        \item Any two $m$s cross at most once
    \end{itemize}
\end{lem}

Here is the bijection in the reverse direction. 
\begin{defn}
Let $D$ be a valid $m$-diagram (i.e. satisfying Lemma \ref{lem:ms_conditions}) with boundary vertices $\{1,\dots,3n\}$ consisting of $n$ disjoint $m$’s, each written $(i,j,k)$ with $i<j<k$ (first arc $(i,j)$, second arc $(j,k)$).  
Construct a $3\times n$ tableau $\Psi(D)$ by scanning $t=1,\dots,3n$ from left to right:
\[
t\ \mapsto\ 
\begin{cases}
\text{row $1$}, & \text{if $t$ is the start $i$ of some $m$},\\
\text{row $2$}, & \text{if $t$ is the middle $j$ of some $m$},\\
\text{row $3$}, & \text{if $t$ is the end $k$ of some $m$}.
\end{cases}
\]
Entries are appended in order within each row, producing a valid standard Young tableaux of shape $3\times n$. Note that the numbers to the left and below $t$ in $\Psi(D)$ are smaller than $t$.
\end{defn}

We will refer to the $3n$ points on the boundary as steps. They will later correspond to certain steps in a lattice path. Now we review how to produce a web from a $m$-diagram.

\begin{defn}
Let $D$ be an $m$-diagram arising from a standard Young tableaux of shape $3\times n$.
The associated reduced $\mathfrak{sl}_3$ web is obtained by the following steps:
\begin{enumerate}
    \item For each $m=(i,j,k)$, the boundary vertex $j$ has degree two in $D$. Replace $j$ by a ``Y'' vertex, which joins $(i,j)$ and $(j,k)$ at a common interior trivalent vertex.  

    \item 
    Direct both arcs in each $m$ from the boundary toward its trivalent vertex (produced in step $1$).

    \item
    Each crossing of two arcs is a degree-four interior vertex. Replace it by two trivalent vertices connected by a short edge, in the unique way that preserves the given edge orientations.

\end{enumerate}
The resulting planar directed graph has all internal vertices trivalent, all boundary vertices of degree one (edges oriented away from boundary), and every vertex a source or a sink.
\end{defn}

\FloatBarrier

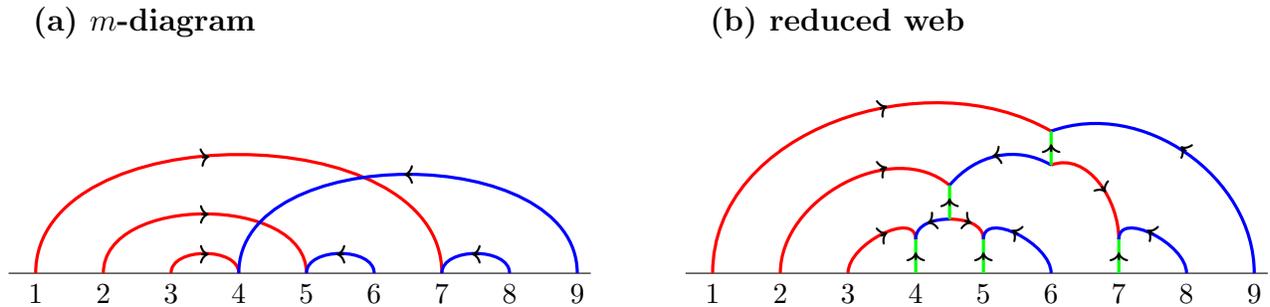
\begin{figure}[ht]
\centering
\begin{tikzpicture}[scale=0.9, every node/.style={font=\small}]
  \begin{scope}[shift={(0,0)}]
    \node[anchor=west,font=\bfseries] at (0.8,3.7) {(a) $m$-diagram};

    \draw (0.6,0) -- (9.2,0);
    \foreach \i in {1,...,9} {
      \draw (\i,0) -- (\i,0.08);
      \node[below] at (\i,0) {\i};
    }

    \begin{scope}[very thick,red]
      \draw (3,0) to[out=90,in=90] (4,0);
      \draw (2,0) to[out=90,in=90] (5,0);
      \draw (1,0) to[out=90,in=90] (7,0);
    \end{scope}

    \begin{scope}[very thick,blue]
      \draw (7,0) to[out=90,in=90] (8,0);
      \draw (5,0) to[out=90,in=90] (6,0);
      \draw (4,0) to[out=90,in=90] (9,0);
    \end{scope}

    \node at (3.5,0.28) {\tikz \draw[->,black,thick,rotate=0] (0,0) -- (0.1,0);};
    \node at (3.5,0.87) {\tikz \draw[->,black,thick,rotate=0] (0,0) -- (0.1,0);};
    \node at (3.5,1.71) {\tikz \draw[->,black,thick,rotate=7] (0,0) -- (0.1,0);};

    \node at (5.5,0.28) {\tikz \draw[->,black,thick,rotate=180] (0,0) -- (0.1,0);};
    \node at (7.5,0.28) {\tikz \draw[->,black,thick,rotate=180] (0,0) -- (0.1,0);};
    \node at (6.5,1.46) {\tikz \draw[->,black,thick,rotate=180] (0,0) -- (0.1,0);};
  \end{scope}

  \begin{scope}[shift={(10,0)}]
    \node[anchor=west,font=\bfseries] at (0.8,3.7) {(b) reduced web};
    \draw (0.6,0) -- (9.2,0);
    \foreach \i in {1,...,9} {
      \draw (\i,0) -- (\i,0.08);
      \node[below] at (\i,0) {\i};
    }

    \begin{scope}[very thick,red]
      \draw (3,0) to[out=90,in=90] (4,0.5);
      \draw (2,0) to[out=90,in=140] (4.5,1.3);
      \draw (5,0.5) to[out=90,in=0] (4.5,0.8);
      \draw (1,0) to[out=90,in=150] (6,2.1);
      \draw (6,1.6) to[out=20,in=90] (7,0.5);
    \end{scope}
    
    \begin{scope}[very thick,blue]
      \draw (7,0.5) to[out=90,in=90] (8,0);
      \draw (5,0.5) to[out=90,in=90] (6,0);
      \draw (4,0.5) to[out=90,in=180] (4.5,0.8);
      \draw (9,0) to[out=90,in=20] (6,2.1);
      \draw (4.5,1.3) to[out=50,in=150] (6,1.6);
    \end{scope}

    \begin{scope}[very thick,green]
      \draw (7,0) to[out=90,in=270] (7,0.5);
      \draw (5,0) to[out=90,in=270] (5,0.5);
      \draw (4,0) to[out=90,in=270] (4,0.5);
      \draw (4.5,0.8) to[out=90,in=270] (4.5,1.3);
      \draw (6,1.6) to[out=90,in=270] (6,2.1);
    \end{scope} 
    \node at (3.5,0.55) {\tikz \draw[->,black,thick,rotate=25] (0,0) -- (0.1,0);};
    \node at (4,0.25) {\tikz \draw[->,black,thick,rotate=90] (0,0) -- (0.1,0);};
    \node at (5,0.25) {\tikz \draw[->,black,thick,rotate=90] (0,0) -- (0.1,0);};
    \node at (7,0.25) {\tikz \draw[->,black,thick,rotate=90] (0,0) -- (0.1,0);};
    \node at (4.5,1.05) {\tikz \draw[->,black,thick,rotate=90] (0,0) -- (0.1,0);};
    \node at (6,1.85) {\tikz \draw[->,black,thick,rotate=90] (0,0) -- (0.1,0);};
    \node at (3.5,1.53) {\tikz \draw[->,black,thick,rotate=20] (0,0) -- (0.1,0);};
    \node at (3.5,2.43) {\tikz \draw[->,black,thick,rotate=20] (0,0) -- (0.1,0);};
    \node at (4.25,0.78) {\tikz \draw[->,black,thick,rotate=200] (0,0) -- (0.1,0);};
    \node at (4.75,0.78) {\tikz \draw[->,black,thick,rotate=-20] (0,0) -- (0.1,0);};
    \node at (5.5,0.58) {\tikz \draw[->,black,thick,rotate=150] (0,0) -- (0.1,0);};
    \node at (7.5,0.58) {\tikz \draw[->,black,thick,rotate=150] (0,0) -- (0.1,0);};
    \node at (6.74,1.28) {\tikz \draw[->,black,thick,rotate=305] (0,0) -- (0.1,0);};
    \node at (5.2,1.75) {\tikz \draw[->,black,thick,rotate=190] (0,0) -- (0.1,0);};
    \node at (8,1.8) {\tikz \draw[->,black,thick,rotate=140] (0,0) -- (0.1,0);};
    \end{scope}
\end{tikzpicture}
\caption{Example of the map from an $m$-diagram (a) to its associated reduced $\mathfrak{sl}_3$ web (b) for a $3\times 3$ tableau. Blue arcs join second-row entries to first-row entries; red arcs join third-row entries to second-row entries. Replacing middle vertices with trivalent Y’s, orienting edges, and resolving crossings yields the reduced web. Arrows denote orientations.}
\label{fig:m-diagram-to-web}
\end{figure}

As seen in Figure \ref{fig:m-diagram-to-web} the web can be obtained by adding a short vertical green edge at each crossing and each middle of an $m$ along the boundary. The red and blue edges have the same orientations in both figures, and the green edges are oriented so that each vertex in the web is a source or sink.

We also see that with the given coloring, each vertex in the web is incident to three edges that all have different colors.

\subsection{Interior Faces in an \texorpdfstring{$m$}{m}-diagram}

In the planar embedding of an $m$-diagram, the arcs divide the upper half–plane into connected regions called \emph{faces}. Then each face in the $m$-diagram becomes a face in the corresponding web. A face is called \emph{interior} if it is not incident to the boundary line (i.e., it is completely enclosed by arcs). 

\begin{lem}\label{lem:face-size-drop-by-two}
Let $D$ be an $m$-diagram, and let $W$ be the associated reduced web. Let $F$ be an interior face of $W$ whose boundary has $2k$ edges, and let $\tilde F$ denote the associated interior face in $D$. Let $p$ and $q$ denote the leftmost and rightmost arc crossings incident to the face $\tilde F$. 
\begin{itemize}
    \item[(1)] Then $p$ and $q$ and are connected by two red/blue alternating arc segment paths, which bound the face $\tilde F$ (one is above the face and one is below).
    \item[(2)] The path above the face $\tilde F$ has one or two arc segments.
    \item[(3)] The total number of arc segments in the two paths is $2k-2$.
\end{itemize}
See Figure \ref{fig:m-diagram-to-web_2}.
\end{lem}

\FloatBarrier

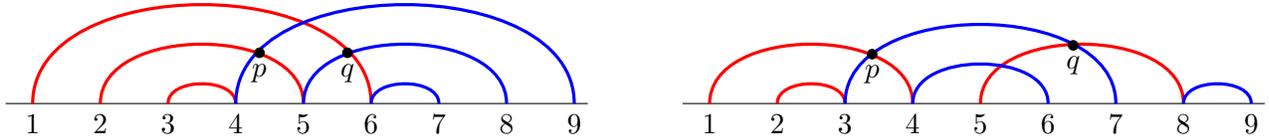
\begin{figure}[ht]
\centering
\begin{tikzpicture}[scale=0.9, every node/.style={font=\small}]
  \begin{scope}[shift={(0,0)}]
    \node[anchor=west,font=\bfseries] at (0.8,3.7) {};

    \draw (0.6,0) -- (9.2,0);
    \foreach \i in {1,...,9} {
      \draw (\i,0) -- (\i,0.08);
      \node[below] at (\i,0) {\i};
    }

    \begin{scope}[very thick,red]
      \draw (3,0) to[out=90,in=90] (4,0);
      \draw (2,0) to[out=90,in=90] (5,0);
      \draw (1,0) to[out=90,in=90] (6,0);
    \end{scope}

    \begin{scope}[very thick,blue]
      \draw (7,0) to[out=90,in=90] (6,0);
      \draw (5,0) to[out=90,in=90] (8,0);
      \draw (4,0) to[out=90,in=90] (9,0);
    \end{scope}
    \filldraw[black] (4.35,0.75) circle (2pt);
    \filldraw[black] (5.65,0.75) circle (2pt);
    \node at (4.35,0.45) {$p$};
    \node at (5.65,0.45) {$q$};
  \end{scope}

  \begin{scope}[shift={(10,0)}]
    \node[anchor=west,font=\bfseries] at (0.8,3.7) {};

    \draw (0.6,0) -- (9.2,0);
    \foreach \i in {1,...,9} {
      \draw (\i,0) -- (\i,0.08);
      \node[below] at (\i,0) {\i};
    }

    \begin{scope}[very thick,red]
      \draw (1,0) to[out=90,in=90] (4,0);
      \draw (2,0) to[out=90,in=90] (3,0);
      \draw (5,0) to[out=90,in=90] (8,0);
    \end{scope}

    \begin{scope}[very thick,blue]
      \draw (3,0) to[out=90,in=90] (7,0);
      \draw (4,0) to[out=90,in=90] (6,0);
      \draw (8,0) to[out=90,in=90] (9,0);
    \end{scope}
    \filldraw[black] (3.4,0.73) circle (2pt);
    \filldraw[black] (6.37,0.86) circle (2pt);
    \node at (3.4,0.43) {$p$};
    \node at (6.37,0.56) {$q$};
  \end{scope}
\end{tikzpicture}
\caption{Interior faces of size $4$ in a $m$-diagram (corresponding to $6$ in the reduced web). Black vertices $p, q$ denote the leftmost and rightmost crossings, which expand to green edges in the web. There are either one or two arc segments above the face.}
\label{fig:m-diagram-to-web_2}
\end{figure}

\FloatBarrier

\begin{proof}
(1) can be seen since only arcs of different colors intersect. In an $m$-diagram, every red/blue arc is drawn as a curve that is the graph of a strictly concave function on the interval between its endpoints. From this we conclude that the arc segment path connecting $p$ and $q$ above $\tilde F$ must have one or two total arc segments as shown in Figure \ref{fig:m-diagram-to-web}. This proves $(2)$. Below the face we have an alternating sequence of red/blue arc segments. Now consider $F$; by concavity, the only two green edges incident to $F$ must be the expanded crossings at the points $p$ and $q$. This means there are $2k-2$ red/blue edges bounding $F$, which means there are $2k-2$ arc segments bounding $\tilde F$. This proves $(3)$.
\end{proof}

\section{Crossing and Face Probabilities} \label{sec:crossing_probabilities}
In this section we use the lattice path model to assign probabilities to 
crossing events in an $m$--diagram. By encoding open arcs as coordinates of the 
walk, each local crossing pattern becomes an axis–hitting problem for a 
two-dimensional random walk. This framework allows us to compute probabilities 
for arc crossings and, in turn, for the occurrence of specific face types.

We consider the uniform measure on the reduced webs in $W_n$.  
As reviewed in \cite{tymoczko2007webs}, this is equivalent to taking the uniform measure 
on $3\times n$ standard Young tableaux and considering their associated webs.  
Each $3\times n$ SYT corresponds bijectively to a lattice path in the upper right quadrant,
\[
S_t = (A_t,B_t) \in \mathbb{Z}_{\ge 0}^2, \qquad t\in [0, 3n]
\]
with allowed steps
\[
s_1 = (1,0),\qquad s_2 = (-1,1),\qquad s_3 = (0,-1),
\]
starting and ending at the origin $(0,0)$.  
The bijection is as follows: read the entries $1,\dots,3n$ of the SYT in increasing order.  
If $i$ lies in the first row, take step $s_1$ at time $i$;  
if $i$ lies in the second row, take step $s_2$;  
and if $i$ lies in the third row, take step $s_3$.  
Because the tableaux is standard, the resulting path never leaves the nonnegative quadrant, 
and it necessarily returns to $(0,0)$ after $3n$ steps. Therefore we can consider the uniform measure on such paths to generate a uniform standard Young tableaux.

We recall Theorem 6 (for steps in directions $(1,0)$, $(-1,1)$ and $(0,-1)$) of \cite{kenyon2015bipolar} regarding the limiting behavior of the lattice path.
\begin{thm}\label{thm:iid_walk}
    As $n\to\infty$, the scaled walk $S_{\lfloor 3n t\rfloor}/\sqrt{3n}$ converges in law (weakly w.r.t the $L^{\infty}$ norm on $[0,1]$ to the Brownian excursion in the nonnegative quadrant starting and ending at the origin, with covariance matrix $\begin{pmatrix}
        2/3 & -1/3\\
        -1/3 & 2/3
    \end{pmatrix}$

    Furthermore, the walk is locally approximately i.i.d.: For any $\epsilon_1>0$ there is an $\epsilon_2>0$ so that as $n\to\infty$, for any sequence of $\epsilon_2n$ consecutive moves that is disjoint from the first or last $\epsilon_1n$ moves, the $\epsilon_2n$ moves are within total variation distance $\epsilon_1$ from an i.i.d sequence, in which the moves $(1,0)$, $(-1,1)$ and $(0,-1)$ each occur with probability $1/3$.
\end{thm}

We will use the locally approximately i.i.d property to calculate the probability of certain events in the reduced web.

\subsection{Bijection from subpaths and partial SYT to partial \texorpdfstring{$m$}{m}-diagram}

We now explain how the initial segment of the lattice path, or equivalently the partial tableaux
filled with $1,\dots,t$, corresponds to a partially completed $m$-diagram, and how each new step
in the path updates this diagram in a consistent way. We start at time $0$ with an empty tableaux and empty $m$-diagram (no arcs open).  
Fix $t\ge 1$ and consider the prefix tableaux consisting of the entries $\{1,\dots,t\}$ of $T$ in
their rows (deleting larger entries). Reading $1,2,\dots,t$ in order produces a prefix of the walk
$(A_s,B_s)_{s=0}^t$ with steps
\[
\text{row 1}\ \leftrightarrow\ (1,0),\qquad
\text{row 2}\ \leftrightarrow\ (-1,1),\qquad
\text{row 3}\ \leftrightarrow\ (0,-1),
\]
and, simultaneously, a local piece of the $m$-diagram on boundary steps $\{1,\dots,t\}$.
We build this piece step by step as follows.

Maintain a single stack $\mathsf{S}$ of currently open arcs, where each entry is a pair $(j, C)$
with $j$ the boundary step where the arc started and $C \in \{R,B\}$ its color.
Initialize $\mathsf{S}$ as empty at time $s=0$.

At each time $s = 1, \dots, t$:
\begin{itemize}
  \item \textbf{Row 1 (step $(1,0)$):} Open a red arc at boundary step $s$ and push $(s, R)$ onto the bottom of $\mathsf{S}$.
  \item \textbf{Row 2 (step $(-1,1)$):} Pop the bottom open red arc $(j, R)$ from $\mathsf{S}$ and draw the red semicircle $(j, s)$,
        closing the lowest open red arc. This makes $s$ the middle of an $m$.
        Then open a blue arc at $s$ and push $(s, B)$ onto the bottom of $\mathsf{S}$.
  \item \textbf{Row 3 (step $(0,-1)$):} Pop the bottom open blue arc $(k, B)$ from $\mathsf{S}$ and draw the blue semicircle $(k, s)$,
        closing the lowest open blue arc.
\end{itemize}

\noindent\textit{Example.} If the first five steps place entries in rows $1,2,1,3,2$, then the stack evolves as 
$[(1,R)], [(2,B)], [(2,B), (3, R)], [(3,R)], [(5,B)]$.

\bigskip

Arcs of the same color never cross because we always attach to the most recent unmatched start, while red/blue crossings arise precisely when red/blue arc is closed and there are arcs of the other color open under the closed arc.

Let $R_s$ and $B_s$ be the numbers of open red/blue arcs after time $s$.
These evolve by
\[
\begin{array}{c|ccc}
\text{move at }s & \text{row 1} & \text{row 2} & \text{row 3}\\\hline
\Delta R_s & +1 & -1 & 0\\
\Delta B_s & 0  & +1 & -1
\end{array}
\]
with $(R_0,B_0)=(0,0)$. Thus $R_s=\#\{\text{row 1 up to }s\}-\#\{\text{row 2 up to }s\}$ and
$B_s=\#\{\text{row 2 up to }s\}-\#\{\text{row 3 up to }s\}$. In particular, the prefix tableau on $\{1,\dots,t\}$
determines the local $m$-diagram between boundary steps $1$ and $t$ together with the number of ``dangling''
open arcs recorded by $(R_t,B_t)$. See Figure \ref{fig:3x3-tableau-mdiagram-part} for an example.

\begin{remark}
    In future computations we will start subpaths at $(R_0,B_0)=(1,1)$ so that when either $R_0=0$ or $B_0=0$ the path hits one of the two axes.
\end{remark}

\FloatBarrier

\begin{figure}[ht]
\centering
\begin{tikzpicture}[scale=0.95, every node/.style={font=\small}]
  \begin{scope}[shift={(0,0)}]
    \node[anchor=west,font=\bfseries] at (-0.2,3.7) {(a) Prefix tableaux};
    \foreach \x in {0,1,2} {
      \foreach \y in {0} {
        \draw (\x,\y) rectangle ++(1,1);
      }
    }
    \foreach \x in {0,1} {
      \foreach \y in {1} {
        \draw (\x,\y) rectangle ++(1,1);
      }
    }
    \node at (0.5,0.5) {1};
    \node at (1.5,0.5) {2};
    \node at (2.5,0.5) {3};
    \node at (0.5,1.5) {4};
    \node at (1.5,1.5) {5};
  \end{scope}

  \begin{scope}[shift={(6.2,0)}]
    \node[anchor=west,font=\bfseries] at (-0.2,3.7) {(b) partial $m$-diagram};

    \draw (0.8,0) -- (9.2,0);
    \foreach \i in {1,...,9} {
      \draw (\i,0) -- (\i,0.08);
      \node[below] at (\i,0) {\i};
    }

    \begin{scope}[very thick,red]
      \draw (3,0) to[out=90,in=90] (4,0);
      \draw (2,0) to[out=90,in=90] (5,0);
      \draw (1,0) to[out=90,in=180] (5.3,1.8);
    \end{scope}

    \begin{scope}[very thick,blue]
      \draw (5,0) to[out=90,in=180] (5.3,0.2);
      \draw (4,0) to[out=90,in=180] (5.3,1);
    \end{scope}

    \draw[red,very thick] (0.4,3.2)--(1.2,3.2) node[right,black] {first arcs (row 2 $\to$ row 1)};
    \draw[blue,very thick]  (0.4,2.8)--(1.2,2.8) node[right,black] {second arcs (row 3 $\to$ row 2)};
  \end{scope}
\end{tikzpicture}
\caption{A Prefix Young tableaux (a) and its partial $m$-diagram (b). Consider the step $6$. If $6$ was added to the bottom row we open a red arc at step $6$. If $6$ went in the middle row, we close the highest red arc (which would cross the two open blue arcs) and add an open blue arc at step $6$. If $6$ went in the top row, the bottom blue arc would be closed. The stack is $\mathsf{S}=[(5, B), (4, B), (1, R)]$}
\label{fig:3x3-tableau-mdiagram-part}
\end{figure}
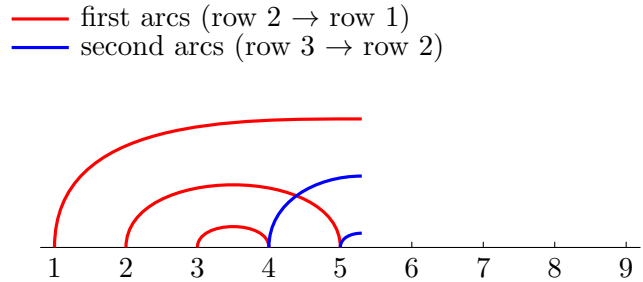

\FloatBarrier

\subsection{Calculating crossing Probabilities}
In this section, we calculate probabilities of certain crossing events in the $m$-diagram for the web. Consider a red arc connecting the steps $(r_1, r_2)$ on the boundary. We say that a blue connecting the steps $(b_1, b_2)$ arc crosses the red arc from the left if $b_1<r_1<b_2<r_2$, and crosses the arc from the right if $r_1<b_1<r_2<b_2$ (and vice versa).

At any point in time there are some (possibly $0$) red arcs and blue arcs that are opened, but not yet closed i.e. still on the stack $\mathsf{S}$. Suppose have one open red arc $R$ on $\mathsf{S}$ as in Figure \ref{fig:initial}, so we have just made a step in the $(1,0)$ direction. Then we will calculate the probability that the next crossing event with $R$ is one of the following three crossings events: $(1)$ $R$ does not cross any more blue arcs \ref{fig:event_1}, $(2)$ a blue arc $B$ crosses $R$ arc from the left \ref{fig:event_2}, $(3)$ the $R$ crosses $k$ more blue arcs from the left \ref{fig:event_3}.

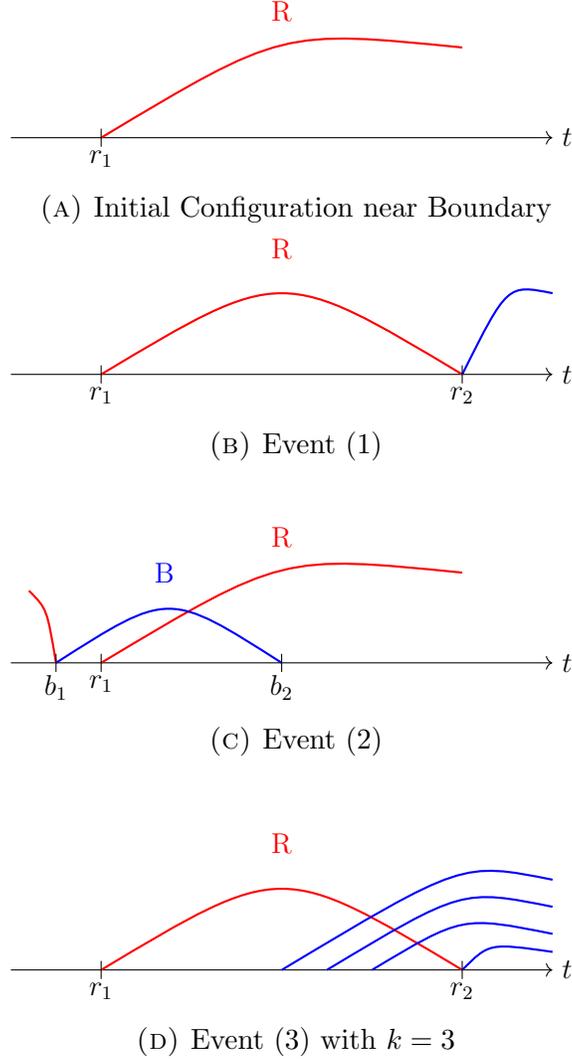
\begin{figure}[ht]
\begin{center}
\begin{subfigure}{\textwidth}
\centering
\begin{tikzpicture}[scale=1.2, every node/.style={scale=0.9}]
  \draw[->] (0,0) -- (6,0) node[right] {$t$};

  \coordinate (r1) at (1,0);
  \coordinate (r2) at (5, 1);

  \draw[red, thick] (r1) .. controls (3,1.2) .. (r2);
  \node[red] at (3,1.4) {R};

  \foreach \x/\label in {1/$r_1$}
    \draw (\x,-0.1) -- (\x,0.1) node[below=5pt] {\label};

\end{tikzpicture}
\caption{Initial Configuration near Boundary}
\label{fig:initial}
\end{subfigure}
\bigskip
\begin{subfigure}{\textwidth}
\centering
\begin{tikzpicture}[scale=1.2, every node/.style={scale=0.9}]
  \draw[->] (0,0) -- (6,0) node[right] {$t$};

  \coordinate (r1) at (1,0);
  \coordinate (r2) at (5,0);
  \coordinate (b1) at (6, 0.9);

  \draw[red, thick] (r1) .. controls (3,1.2) .. (r2);
  \node[red] at (3,1.4) {R};
  \draw[blue, thick] (r2) .. controls (5.5,1) .. (b1);

  \foreach \x/\label in {1/$r_1$, 5/$r_2$}
    \draw (\x,-0.1) -- (\x,0.1) node[below=5pt] {\label};

\end{tikzpicture}
\caption{Event $(1)$}
\label{fig:event_1}
\end{subfigure}

\bigskip
\begin{subfigure}{\textwidth}
\centering
\begin{tikzpicture}[scale=1.2, every node/.style={scale=0.9}]
  \draw[->] (0,0) -- (6,0) node[right] {$t$};

  \coordinate (r1) at (1,0);
  \coordinate (r2) at (5,1);
  \coordinate (b1) at (0.5,0);
  \coordinate (b2) at (3,0);
  \coordinate (r3) at (0.2,0.8);
  \coordinate (r4) at (0.5,0);

  \draw[red, thick] (r1) .. controls (3,1.2) .. (r2);
  \node[red] at (3,1.4) {R};

  \draw[blue, thick] (b1) .. controls (1.75,0.8) .. (b2);
  \node[blue] at (1.7,1) {B};

  \draw[red, thick] (r3) .. controls (0.4,0.6) .. (r4);

  \foreach \x/\label in {0.5/$b_1$, 1/$r_1$, 3/$b_2$}
    \draw (\x,-0.1) -- (\x,0.1) node[below=5pt] {\label};

\end{tikzpicture}
\caption{Event (2)}
\label{fig:event_2}
\end{subfigure}
\bigskip

\begin{subfigure}{\textwidth}
\centering
\begin{tikzpicture}[scale=1.2, every node/.style={scale=0.9}]
  \draw[->] (0,0) -- (6,0) node[right] {$t$};

  \coordinate (r1) at (1,0);
  \coordinate (r2) at (5,0);
  \coordinate (b1) at (3,0);
  \coordinate (b2) at (6,1);
   \coordinate (b3) at (3.5,0);
  \coordinate (b4) at (6,0.7);
  \coordinate (b5) at (4,0);
  \coordinate (b6) at (6,0.4);
  \coordinate (b7) at (5,0);
  \coordinate (b8) at (6,0.2);

  \draw[red, thick] (r1) .. controls (3,1.2) .. (r2);
  \node[red] at (3,1.4) {R};

  \draw[blue, thick] (b1) .. controls (5,1.2) .. (b2);
  \draw[blue, thick] (b3) .. controls (5,0.9) .. (b4);
  \draw[blue, thick] (b5) .. controls (5,0.6) .. (b6);
  \draw[blue, thick] (b7) .. controls (5.3,0.3) .. (b8);

  \foreach \x/\label in {1/$r_1$, 5/$r_2$}
    \draw (\x,-0.1) -- (\x,0.1) node[below=5pt] {\label};

\end{tikzpicture}
\caption{Event (3) with $k=3$}
\label{fig:event_3}
\end{subfigure}
\end{center}
\caption{Crossing Events}
\end{figure}

\FloatBarrier

We will rewrite the probability of those three events in terms of certain discrete harmonic functions $h_a(x, y)$ and $g(z)$. 

\begin{defn}
We define $h_a(x, y)$ as the solution to the discrete Laplace equation $\Delta h_a(x, y)=0$ in the upper right quadrant of $\mathbb Z^2$ with boundary conditions $\delta_a(x, y)$ on the axes $x=0$ or $y=0$ (i.e. $h(x, y)=1$ when $a=(x, y)$ and $0$ otherwise) and \[\Delta f (x, y)=\frac{1}{3}\left[f(x+1, y)+f(x-1, y+1)+f(x, y-1)-3f(x, y)\right]\] $g(x, y)$ is a solution to the same Laplace equation, but a different boundary condition. The boundary conditions are $g(0, y)=0$ and $g(x, 0)=1$. 
\end{defn}

The explicit values of $h_a(z)$ and $g(z)$ are given in Proposition \ref{prop:point_mass_solution} and Lemma \ref{lem:full_boundary_g} by solving the corresponding Dirichlet problems, and obtaining integral formulas. See Example \ref{ex:h(0,2)11} for the computation of $h_{(0,2)}(1,1)$.

\begin{prop} \label{prop:crossing_probabilities_red}
    Suppose we start from step $r_1\in [\epsilon_1 n, (1-\epsilon_1-\epsilon_2)n]$ (i.e. in the interval where local i.i.d. property holds) and suppose at $r_1$ we opened a red arc $R$ (i.e. right after a Type $1$ move \ref{fig:initial}). Then we have the following probabilities of the following three events as $n\to\infty$:
    \begin{itemize}
        \item[(1)] The probability that $R$ does not cross any more blue arcs is $h_{(0,2)}(1,1)=\frac{243\sqrt{3}}{40\pi}-3\approx 0.3493$. See \ref{fig:event_1}
        
        \item[(2)] The probability that $R$ is crossed by another blue arc $B$ from the left whose second endpoint is after the first endpoint of the red arc $R$ is $g(1,1)$. See \ref{fig:event_2}
        \item[(3)] The probability that $R$ crosses another $k$ arcs from the left is $h_{(0, k+2)}(1, 1)$. See \ref{fig:event_3}
    \end{itemize}

\end{prop}

\begin{proof}
    Consider the subpath of the lattice walk corresponding to the portion of the web starting from $r_1$. Suppose the walk starts at $X_0=(1,1)$ in $\mathbb Z_{\geq 0}^2$ and takes the steps
    \begin{itemize}
        \item $(1,0)$: open new red arc
        \item $(-1,1)$: close red arc that was opened last and open new blue arc
        \item $(0,-1)$: close the blue arc that was opened last
    \end{itemize}

    Let $X_t=(A_t, B_t)$ denote the location of the walk after $t$ steps. Then $A_t-1$ is the number of open red arcs below (by below we mean opened later) $R$, and $B_t-1$ is the number of blue arcs below $R$. Define $\tau=\inf\{t\geq 0:A_t=0 \textrm{ or } B_t=0\}$. We note that,
\begin{itemize}
    \item If $ A_\tau = 0 $, then $ R $ closes time $\tau$. If $ B_\tau = k $, then $ R $ crossed $ k - 2 $ blue arcs. Note that $ R $ crosses $ k-2 $ blue arcs rather than $ k-1 $, because at time $ \tau $ a blue arc is opened simultaneously as $ R $ closes, which shifts the count by $1$.

    \item If $ B_\tau = 0 $, then the move $\tau$ closed a blue arc that crossed $ R $ from the left (event (2)).
\end{itemize}

The three possible crossing events only occur when the walk hits one of the two axes, so therefore suffices to analyze the path up to $ \tau $.

Since the walk is locally approximately i.i.d.\ (Theorem~\ref{thm:iid_walk}), for any $ \epsilon_1 > 0 $ we can choose $ \epsilon_2$ so that any $\epsilon_2n$ steps (disjoint from the first or last $ \epsilon_1 n $ steps) are within total variation distance $ \epsilon_1 $ of the i.i.d.\ walk where each move is $ (1,0) $, $ (-1,1) $, or $ (0,-1) $ with probability $ 1/3 $.

Moreover, the i.i.d.\ walk hits the boundary in finite time almost surely:
\[
\mathbb{P}(\tau_{\mathrm{iid}}<\infty)=1.
\]
For any $T$ we have $\mathbb{P}(\tau_{\mathrm{iid}}>T)\to 0$ as $T\to\infty$. 
By coupling with the locally approximately i.i.d.\ segment, for any $\epsilon_1>0$,
\[
\mathbb{P}(\tau > T)\;\le\;\mathbb{P}(\tau_{\mathrm{iid}}>T)+\epsilon_1.
\]
Letting $T\to\infty$ (with $T=o(\epsilon_2 n)$), we conclude $\mathbb{P}(\tau=\infty)=0$.

Thus, the walk hits the boundary in finite time almost surely, we may apply the locally approximately i.i.d.\ property to analyze the distribution of the path up to $ \tau $.

    Therefore, the probability of each event corresponds to the probability that the walk from $ (1,1) $ first hits the appropriate boundary:
\begin{itemize}
    \item Event (1): hit $ A=0 $ at $ (0,2) $, so the probability is $ h_{(0,2)}(1,1) $.
    \item Event (2): hit $ B=0 $, so the probability is $ g(1,1) $ where $ g(x,y) $ solves $ \Delta g = 0 $ with $ g(x,0)=1 $, $ g(0,y)=0 $.
    \item Event (3): hit $ A=0 $ at $ (0,k+2) $, so the probability is $ h_{(0,k+2)}(1,1) $.
\end{itemize}

Here, $ h_a(x,y) $ solves
\[
\Delta h_a(x,y) = 0
\]
where
\[
\Delta f(x,y) = \frac{1}{3}\big( f(x+1,y) + f(x-1,y+1) + f(x,y-1) \big) - f(x,y)
\]
in the interior, with boundary condition $ h_a(x,y)=\delta_a(x,y) $ on $ x=0 $ or $ y=0 $.
\end{proof}

\begin{prop} \label{prop:crossing_probabilities_blue}
    Suppose we start from step $b_1\in [\epsilon_1 n, (1-\epsilon_1-\epsilon_2)n]$ and suppose at $b_1$ we have just opened a blue arc $B$ (i.e. right after a Type $2$ move). We calculate the probability of the following three events:
    \begin{itemize}
        \item[(1)] The probability that $B$ does not cross any more red arcs is $h_{(1,0)}(1,1)$
        
        \item[(2)] The probability that $B$ is crossed by red arc $R$ from the left whose second step is after the first step of the blue arc $B$ is $1-g(1,1)$
        \item[(3)] The probability that $B$ crosses another $k$ arcs from the left is $h_{(k+1, 0)}(1, 1)$
    \end{itemize}
\end{prop}

\begin{proof}
    The proof proceeds analogously to Proposition~\ref{prop:crossing_probabilities_red}, with the roles of red and blue arcs reversed. The walk starts at $(1,1)$, and the axis-hitting probabilities correspond to closing $ B $ without further crossings (event (1)), being crossed by a red arc (event (2)), or crossing additional arcs (event (3)). The hitting probabilities follow from the same discrete Laplace equation with appropriate boundary conditions.

\end{proof}

\FloatBarrier

\begin{figure}[ht]
\begin{center}
\begin{tikzpicture}[scale=1.2, every node/.style={scale=0.9}]
  \draw[->] (0,0) -- (6,0) node[right] {$t$};

  \coordinate (r1) at (1,0);
  \coordinate (r2) at (5.6, 1);
  \coordinate (r3) at (2, 0);
  \coordinate (r4) at (3,0);
  \coordinate (b1) at (3, 0);
  \coordinate (b2) at (4,0);

  \coordinate (r5) at (1.5, 0);
  \coordinate (r6) at (5.6,0.9);
  \coordinate (r7) at (4.3, 0);
  \coordinate (r8) at (5.3,0);
  \coordinate (b3) at (5.3, 0);
  \coordinate (b4) at (5.6,0.8);

  \draw[red, thick] (r1) .. controls (3,1.2) .. (r2);
  \node[red] at (3,1.4) {R};
  \draw[red, thick] (r3) .. controls (2.5,0.4) .. (r4);
  \draw[blue, thick] (b1) .. controls (3.5,0.4) .. (b2);
  \draw[red, thick] (r5) .. controls (3.6,1.1) .. (r6);
  \draw[red, thick] (r7) .. controls (4.8,0.4) .. (r8);
  \draw[blue, thick] (b3) .. controls (5.4,0.4) .. (b4);

  \foreach \x/\label in {1/$r_1$}
    \draw (\x,-0.1) -- (\x,0.1) node[below=5pt] {\label};

\end{tikzpicture}
\end{center}
\caption{Subpath starting at $(2, 2)$}
\label{fig:example_1}
\end{figure}
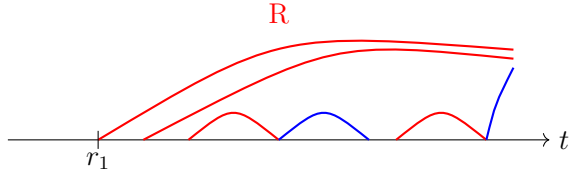

\FloatBarrier

\begin{remark} \label{rem:crossing_starts}
We can generalize Proposition \ref{prop:crossing_probabilities_red} and \ref{prop:crossing_probabilities_blue} to the case where the open arc $R$ or $B$ has other open arcs underneath. For example, in Figure \ref{fig:example_1} we could obtain probabilities analogous to Proposition \ref{prop:crossing_probabilities_red} for crossing events of $R$ by starting the subpath at $(2,2)$ (instead of $(1,1)$), which corresponds to the two additional arcs under $R$.
\end{remark}

\subsection{Calculating face Probabilities}

Recall that by Lemma \ref{lem:face-size-drop-by-two}, any internal face of size $2k$ in a reduced web corresponds to an internal face in the associated $m$-diagram that is surrounded by $2k - 2$ arcs, alternating between red and blue. 

When we construct the web from the $ m $-diagram:
\begin{itemize}
    \item We add a green edge at each red-blue crossing in the $m$-diagram. Call these edges \emph{crossing edges}.
    \item The original  red/blue arcs of the $m$-diagram become \emph{arc edges} in the web.
\end{itemize}

Also by Lemma \ref{lem:face-size-drop-by-two}. Any internal face of size $2k$ in the web contains exactly two crossing edges and $2k - 2$ arc edges, coming from the $2k - 2$ arcs that surround the face, and the two crossing edges are connected by either one or two arc edges that lie \emph{above} the face in the $m$-diagram. The remaining arc edges (either $2k - 2$ or $2k - 3$) connect the crossing edges \emph{below} the face in the $m$-diagram.

\begin{figure}[ht]
\begin{center}
\begin{tikzpicture}[scale=1.2, every node/.style={scale=0.9}]
  \draw[->] (0,0) -- (6,0) node[right] {$t$};

  \coordinate (r1) at (0.5,0);
  \coordinate (r2) at (5.5, 0);
  \coordinate (r3) at (1, 0);
  \coordinate (r4) at (2.5,0);
  \coordinate (r5) at (3.5, 0);
  \coordinate (r6) at (5,0);
   \coordinate (r7) at (3.8, 0);
  \coordinate (r8) at (4.25,0);
  \coordinate (b1) at (2.5,0);
  \coordinate (b2) at (4,0);
  \coordinate (b3) at (5.5,1);
  \coordinate (b4) at (5,0);
  \coordinate (b5) at (0.5,1);
  \coordinate (b6) at (1.5,0);
  \coordinate (b7) at (4.25,0);
  \coordinate (b8) at (4.7,0);

  \draw[red, thick] (r1) .. controls (3,2) .. (r2);
  \draw[red, thick] (r3) .. controls (1.75,0.8) .. (r4);
  \draw[blue, thick] (b1) .. controls (3.25,0.8) .. (b2);
  \draw[blue, thick] (b3) .. controls (5.25,1) .. (b4);
  \draw[red, thick] (r5) .. controls (4.25,0.8) .. (r6);
  \draw[blue, thick] (b5) .. controls (1,0.8) .. (b6);
  \draw[blue, thick] (b7) .. controls (4.475,0.4) .. (b8);
  \draw[red, thick] (r7) .. controls (4.025,0.4) .. (r8);

  \fill[green!50!black] (1.3,0.3) circle (2pt);
  \fill[green!50!black] (2.5,0) circle (2pt);
  \fill[green!50!black] (5,0) circle (2pt);
  \fill[green!50!black] (3.76,0.28) circle (2pt);
\end{tikzpicture}
\caption{Internal face of size $8$ with one arc edge above and $5$ arc edges below.
}
\label{fig:example_face}
\end{center}
\end{figure}
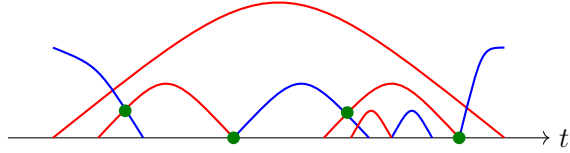

\begin{defn}[Face Type] \label{defn:face_type}
    Let $F$ be an interior face of an $m$-diagram. Let $P(F) = \{ p_1, \dots, p_k \}$ be the ordered sequence of intersection points of arc edges below $F$, listed from left to right. For example, in Figure \ref{fig:example_face} there are four such intersection points marked in green. The size of $ F $ (that is, the number of sides of the face) satisfies:
\[
|F| =
\begin{cases}
|P(F)| + 4 & \text{if } |P(F)| \text{ is even}, \\
|P(F)| + 5 & \text{if } |P(F)| \text{ is odd}.
\end{cases}
\] 
The \textbf{face type} of an internal face $F$ in an $m$-diagram is defined as follows. At each intersection point $p_i$, let $R_i$ and $B_i$ denote the two arcs that cross at $p_i$. For each $i$, consider the segment $A_i$ that starts at $p_i$ and travels along the arc toward the boundary of the diagram in the southeast direction. The value $c_i$ is defined as the number of other arcs that intersect $A_i$ (including the arc $B_i$). The face type of $F$ is then the tuple $(c_1,\dots,c_k)$ together with the color of $A_1$, which is called the starting color and is either $R$ or $B$. For example, in Figure~\ref{fig:example_face}, the face type is $(1,1,2,1), B$.

\end{defn}

Recall that the number of reduced webs of size $n$ is given by the $3$D-Catalan numbers \[|W_n|=\frac{2\cdot(3n)!}{(n+2)!(n+1)!n!}\] 

First we consider near the ends of the path i.e. in the interval $[0,\epsilon_1n)\cup((1-\epsilon_1-2\epsilon_2)n, 0]$. In this case the path is not locally i.i.d., but we show that the number of faces in this case is negligible.

\begin{lem} For each face $F$ let $s_1(F)$ denote the first step of an arc incident to $F$. Then,

\smallskip
(i) For each fixed web and each $s$, there is at most one type $\tau$-face whose first step is $s$.

\smallskip
(ii) Let $A_\tau$ be the total number of type $\tau$ faces across all webs in $W_n$, and let $A_\tau^{\mathrm{mid}}$ be the number of faces $F$ of type $\tau$ such that $s_1(F)\in[\epsilon_1 n,(1-\epsilon_1-2\epsilon_2)n]$. Then
\[
0\;\le\;A_\tau - A_\tau^{\mathrm{mid}}\;\le\; 2(\epsilon_1+\epsilon_2) n\cdot |W_n|.
\]
Consequently, after normalizing by $n\,|W_n|$ as in Theorem~\ref{thm:face_probabilities}, and excluding the first/last $(\epsilon_1+\epsilon_2) n$ steps changes the limiting frequency by $O(\epsilon_1+\epsilon_2)$.
\end{lem}

\begin{proof}
(i) Fix a boundary step $s$ and a type $\tau=(\tau_1,\ldots,\tau_k;C)$. If a type-$\tau$ face $F$ has $s_1(F)=s$, then starting from the arc opened at step $s$ (of color $C$), the successive crossings given by $\tau$ are forced, so arcs of the same color are nested, opposite colors are the only crossings, and any two $m$’s cross at most once. Hence there is either no such face or a unique one, proving the claim.
 
(ii) By (i), each first step in the interval $[0,\epsilon_1 n)\cup ((1-\epsilon_1-2\epsilon_2)n, n]$ contributes at one type $\tau$ face per web. Therefore, the number of type $\tau$ faces whose first step belongs to the boundary interval per web is at most $2(\epsilon_1+\epsilon_2) n$, and across all webs at most $2\varepsilon_1 n\cdot |W_n|$. Dividing by $n\,|W_n|$ shows the normalized contribution of excluded steps is $O(\epsilon_1+\epsilon_2)$. 
\end{proof}

\begin{remark}
The mid-window keeps the first step $s_1(F)$ of face $F$ away from the first $\epsilon_1 n$ steps and last $(\epsilon_1+2\epsilon_2)n$ steps. Thus, all $\epsilon_2 n$–length local windows used in the crossing analysis lie in a region where the walk is locally i.i.d. Note there may be several such windows, since we consider multiple consecutive crossing events. However, these crossing events take finite time a.s. and we have added an additional $\epsilon_2n$, which goes to $\infty$ to account for this. Because each face type involves only finitely many crossing events, the error from the locally i.i.d.\ approximation 
remains $O(\epsilon_1+\epsilon_2)$, so the limiting face probabilities are unaffected.
\end{remark}

Let $ T(\tau, n) $ denote the total number of interior faces of type $ \tau $ among all reduced webs in $ W_n $, where 
\[
\tau = (\tau_1, \dots, \tau_k), \; C
\]
with $ C = R $ or $ B $ denoting the starting color of the face type. In the following proposition we will use $h_a$ and $g$ to denote certain discrete harmonic functions. See Proposition \ref{prop:point_mass_solution} and Lemma \ref{lem:full_boundary_g} for the explicit formulas for $h_a$ and $g$.

\begin{thm}
\label{thm:face_probabilities}
Let $\tilde g_1(\tau_k)=g(\tau_k, 1)$ if $k$ is even and $\tilde g_1(\tau_k)=1-g(\tau_k, 1)$ if $k$ is odd. Similarly, let $\tilde g_2(\tau_k)=1-g(\tau_k, 1)$ if $k$ is even and $\tilde g_2(\tau_k)=g(\tau_k, 1)$ if $k$ is odd. Then, as $ n \to \infty $:
\begin{itemize}
    \item If $ C = R $,
    \[
    \lim_{n \to \infty} \frac{T(\tau, n)}{n \cdot |W_n|} 
    = h_{(0, \tau_1 + 3)}(1,1) \tilde g_1(\tau_k)
      \prod_{\substack{i=2 \\ i \text{ even}}}^{k} h_{(\tau_i + 1, 0)}(1,\tau_{i-1})
      \prod_{\substack{i=2 \\ i \text{ odd}}}^{k} h_{(0, \tau_i + 2)}(\tau_{i-1},1).
    \]
    
    \item If $ C = B $,
    \[
    \lim_{n \to \infty} \frac{T(\tau, n)}{n \cdot |W_n|} 
    = h_{(\tau_1 + 2, 0)}(1,1) \tilde g_2(\tau_k)
      \prod_{\substack{i=2 \\ i \text{ odd}}}^{k} h_{(\tau_i + 1, 0)}(1,\tau_{i-1})
      \prod_{\substack{i=2 \\ i \text{ even}}}^{k} h_{(0, \tau_i + 2)}(\tau_{i-1},1).
    \]
\end{itemize}
\end{thm}

\begin{proof}
Let $A$ be an arc in an $m$-diagram corresponding to a reduced web in $W_n$. By the strong Markov property of $(A_t, B_t)$, the probability of seeing a certain sequence of crossings specified by $\tau$ is given by the product of successive axis hitting probabilities. We compute the probability that $A$ is the first arc edge beneath a face of type $\tau$.

This is done by successively applying Propositions~\ref{prop:crossing_probabilities_red} and~\ref{prop:crossing_probabilities_blue} to determine the probability of the sequence of crossings specified by $\tau$.
\begin{itemize}
    \item The probability for the first arc (depending on whether $C = R$ or $C = B$) accounts for the number of crossings at the first intersection point. This contributes the first factor in the product.
    \item Each subsequent $c_i$ contributes a factor corresponding to the probability of seeing that number of crossings at below the $i$-th intersection. Note that this probability depends on $\tau_{k-1}$, which tells us where to start the next lattice subpath.
    \item The final term involving $g$ accounts for the probability that the last arc edge above the face crosses the last arc edge below the face
\end{itemize}

Note that we evaluate $h$ and $g$ at the points $(1, \tau_{i-1})$ or $(\tau_{i-1}, 1)$ $(i\geq 2)$ to account for the number of open arcs from the previous crossing event that are under the current arc. This is the example in Remark \ref{rem:crossing_starts}.

As an illustration, in Figure~\ref{fig:example_face}, where  $C = B$ and $k = 4$, we begin with the leftmost blue arc. The probability that it crosses two red arcs is $h_{(3,0)}(1,1)$. The next red arc crossing no blue arcs contributes $h_{(0,2)}(1,1)$. The next blue arc crosses two red arcs with probability $h_{(3,0)}(1,1)$. The next red arc crosses no blue arcs with probability
$h_{(0, 2)}(2, 1)$ (note that we start at $(2, 1)$). Finally, the probability that the last blue arc is crossed by the red arc above is $1-g(1,1)$. Taking the product of the probabilities gives the formula above.

\[
h_{(3,0)}(1,1) \cdot h_{(0,2)}(1,1)\cdot h_{(3,0)}(1,1)\cdot  h_{(0, 2)}(2, 1)\cdot (1-g(1,1))
\]
\end{proof}

\section{Green Function} \label{sec:green_function}
To evaluate the crossing probabilities from Section~\ref{sec:crossing_probabilities}, we need 
explicit formulas for the discrete harmonic functions $h_a$ and $g$. These can be 
expressed in terms of Green’s functions on the triangular lattice and in wedge 
domains. This section derives these formulas.

Let $G_\infty(z, z_0)$ denote the Green's function on the infinite equilateral (triangular) lattice, where $z, z_0 \in \mathbb{C}$. Let $G_W(z, z_0)$ denote the Green's function in the wedge $W \subset \mathbb{C}$, which is the 60$^\circ$ sector between the two rays with angles $0$ and $\pi/3$. $\partial W$ consists of the two rays. Then we have:
\[
\Delta_z G_W(z, z_0) = \delta_{z_0}(z) \quad \text{for } z\in W-\partial W, \quad G_W(z, z_0) = 0 \quad \text{for } z \in \partial W.
\]

The allowed steps on the lattice are:
\[
v_1 = 1, \quad v_2 = -\tfrac{1}{2} + \tfrac{\sqrt{3}}{2}i, \quad v_3 = -\tfrac{1}{2} - \tfrac{\sqrt{3}}{2}i.
\]

The discrete Laplacian here is defined as:
\[
(\Delta f)(z) = \frac{1}{3} \sum_{j=1}^3 \left[f(z + v_j) - f(z)\right].
\]
If the argument $z_0$ is omitted, it is understood to be $z_0=0$.

\subsection{Green's Function on the Full Lattice}
\begin{lem} \label{lem:infinite_lattice}
The Green's function on the infinite equilateral (triangular) lattice, written in coordinates 
$z = x e_1 + y e_2$ with 
\[
e_1 = (1,0), \quad e_2 = \left( \frac12, \frac{\sqrt{3}}{2} \right),
\]
is given by
\[
G_\infty(z)=G_\infty(x, y) = \frac{1}{4\pi^2} \int_0^{2\pi} \int_0^{2\pi} 
\frac{e^{i(x\theta + y\phi)}}{\frac13\left(e^{i\theta} + e^{-i\theta + i\phi} + e^{-i\phi}\right) - 1} 
\, d\theta \, d\phi.
\]
\end{lem}

\begin{proof}
In the basis $e_1, e_2$, the eigenfunctions of the discrete Laplacian are
\[
f_{\theta ,\phi}(x, y) = e^{i(\theta x + \phi y)},
\]
and satisfy
\[
\Delta f_{\theta,\phi} = 
\left[ \frac13 \left( e^{i\theta} + e^{-i\theta + i\phi} + e^{-i\phi} \right) - 1 \right] f_{\theta,\phi}.
\]
Then, Fourier inversion gives
\[
G_\infty(x, y) = \frac{1}{4\pi^2} \int_0^{2\pi} \int_0^{2\pi} 
\frac{e^{i(x\theta + y\phi)}}{\frac13\left(e^{i\theta} + e^{-i\theta + i\phi} + e^{-i\phi}\right) - 1} 
\, d\theta \, d\phi,
\]
as claimed.
\end{proof}

\begin{prop}\label{prop:Ginf}
For $(x,y)\in\mathbb{Z}^2, x,y\geq 0$, define the renormalized Green’s function
\[
\mathcal G_\infty(x,y) \;:=\; G_\infty(x,y)-G_\infty(0,0).
\]
Then $\mathcal G_\infty(x,y)$ admits the representation
\[
\mathcal G_\infty(x,y)
= \frac{3}{\pi}\int_{1/4}^1 
\frac{t^y}{\sqrt{\,4t-1\,}}
\Biggl(\sum_{j=1}^x \binom{x}{j}(1-t)^{j-1}\Re\!\big(u(t)^j\big)
- \sum_{\ell=0}^{y-1} t^\ell\Biggr)\,dt,
\]
where 
\[
u(t)=\frac{-1+i\sqrt{\,4t-1\,}}{2t}.
\]
Expanding $\Re(u(t)^j)$ yields a polynomial in $t$ and $\sqrt{\,4t-1\,}$. 
After cancellation with the denominator $\sqrt{\,4t-1\,}$, the integrand reduces to a linear combination of terms $\frac{t^\ell}{\sqrt{4t-1}}$. 
Thus every $\mathcal G_\infty(x,y)$ can be written as a finite $\mathbb{Q}$–linear combination of the integrals
\begin{equation}\label{eq:int_defn}
I_m := \int_{1/4}^1 \frac{t^m}{\sqrt{\,4t-1\,}}\,dt,
\qquad m\in\mathbb{Z}
\end{equation}
Each $I_m$ admits the closed form in $\mathbb{Q}(\sqrt{3},\pi)$ (see Lemma \ref{lem:I_m_values}). When $x$ or $y$ is negative, $\mathcal G_\infty$ can be found by symmetry. Therefore,
\[
\; \mathcal G_\infty(x,y) \in \mathbb{Q}(\sqrt{3},\pi)\quad 
\text{for all $(x,y)\in\mathbb{Z}^2$.}\;
\]
See Figure \ref{fig:explicit_g_inf_wedge} for explicit values.
\end{prop}

\begin{proof}
Start with the formula from Lemma \ref{lem:infinite_lattice}. Let $z=e^{i\theta}$, $w=e^{i\phi}$ and $d\theta=\frac{dz}{iz}$, $d\phi=\frac{dw}{iw}$,
\begin{align*}
G_\infty(x,y)
&=\frac{1}{(2\pi i)^2}\!\oint\!\!\oint
\frac{z^x w^y}{\tfrac13(z+z^{-1}w+w^{-1})-1}\,\frac{dz}{z}\,\frac{dw}{w}\\
&=\frac{3}{(2\pi i)^2}\!\oint_{|w|=1}\!\!\oint_{|z|=1}
\frac{z^x w^y}{z^2w+z(1-3w)+w^2}\,dz\,dw.
\end{align*}
Let $\alpha(w),\beta(w)$ be the roots (in $z$) of $z^2w+z(1-3w)+w^2=0$:
\[
\alpha_{\pm}(w)=\frac{3w-1\pm \sqrt{(1-3w)^2-4w^3}}{2w},\qquad
z^2w+z(1-3w)+w^2=w\,(z-\alpha_+)(z-\alpha_-).
\]
Choose the branch so that $|\alpha_+(w)|<1<|\alpha_-(w)|$ (and $\alpha_\pm(1)=1$) which is possible since $|\alpha_+(w)|\cdot |\alpha_-(w)|=1$.
The $z$–integral is the residue at $z=\alpha_+(w)$:
\[
G_\infty(x,y)
=\frac{3}{(2\pi i)^2}\,(2\pi i)\!\oint_{|w|=1}
\frac{\alpha_+(w)^{\,x}\,w^{\,y}}{w(\alpha_+(w)-\alpha_-(w))}\,dw.
\]
Since $\alpha_+-\alpha_-=\dfrac{\sqrt{(1-3w)^2-4w^3}}{w}$,
\[
G_\infty(x,y)
=-\,\frac{3i}{2\pi}\,\oint_{|w|=1}
\frac{\alpha_+(w)^{\,x}\,w^{\,y}}{\sqrt{(1-3w)^2-4w^3}}\,dw.
\]
Factor
\[
(1-3w)^2-4w^3=(1-w)^2(1-4w),
\]
Then subtracting $G_\infty{(0,0)}$ we obtain
\[
\mathcal G_\infty(x, y)=G_\infty(x, y)-G_\infty(0,0)
=-\,\frac{3i}{2\pi}\,\oint_{|w|=1}
\frac{\alpha_+(w)^{\,x}\,w^{\,y}-1}{(1-w)\sqrt{\,1-4w\,}}\,dw.
\]
Further we have,
\[
\alpha_\pm(w)=\frac{3w-1\pm i(1-w)\sqrt{4w-1}}{2w}=1\pm (1-w)u(w), \qquad u(w)=\frac{-1+i\sqrt{4w-1}}{2w} 
\]
Now we can cancel the $(1-w)$ singularity. Since $x, y\geq 0$ we are left with, 
\[
\mathcal G_\infty(x, y)=-\,\frac{3i}{2\pi}\,\oint_{|w|=1}
\frac{1}{\sqrt{\,1-4w\,}}\left[w^y\sum_{j=1}^x\binom{x}{j}(1-w)^{j-1}u(w)^j-\sum_{\ell=0}^{y-1}w^\ell\right]\,dw.
\]
Where $\sum_{\ell=0}^{y-1}w^\ell$ comes from $\frac{w^y-1}{1-w}$.

Place a branch cut along $[1/4,\infty)$, i.e.\ $\mathrm{Arg}(1-4w)\in(-\pi,\pi)$, and modify the unit circle to run just above/below the cut $[1/4,1]$, with and add small circle around $w=\tfrac14$. See Figure \ref{fig:contour}. The integral around the small circle at $w=1/4$ does not contribute (the integrand is $\sim (w-1/4)^{-1/2}$). 

The contributions from the cut $t\in(1/4,1)$ is,
\[
\mathcal G_\infty(x, y)
=\frac{3}{\pi}\int_{1/4}^{1}
\frac{1}{\sqrt{4t-1}}\left[
t^{\,y}\sum_{j=1}^x\binom{x}{j}(1-t)^{j-1}\Re\!\big(u(t)^{\,j}\big)
-\sum_{\ell=0}^{y-1}t^{\,\ell}\right]dt,
\]
Hence $\mathcal G_\infty(x, y)$ is a $\mathbb{Q}$–linear combination of the $I_m$'s.
\end{proof}

\begin{figure}
\begin{center}
\begin{tikzpicture}[scale=2]
  \draw[->] (-0.5,0)--(2.5,0);
  \draw[->] (1,-1.5)--(1,1.5);

  \draw[thick] (2,0.05) arc[start angle=0+5, end angle=360, radius=1];
  
  \draw[thick] (1.33,0.05)--(2,0.05);
  \draw[thick] (1.33,-0.05)--(2,-0.05);

  \draw[thick] (1.33,0.04) arc[start angle=23,end angle=337,radius=0.1];
  \node at (1.22,0.2) {$w=1/4$};
  \node at (1.25,0) {$\bullet$};
\end{tikzpicture}
\caption{Contour for integration}
\label{fig:contour}
\end{center}
\end{figure}
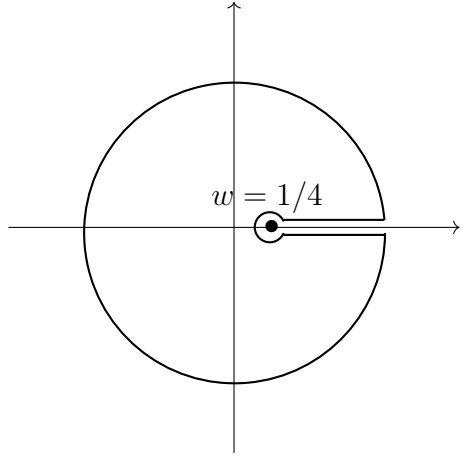
\begin{lem} \label{lem:I_m_values}
The integrals $I_m$ belong to $\mathbb Q(\sqrt3, \pi)$ and their explicit values are presented in the proof.
\end{lem}
\begin{proof}
    Make the substitution $s=\sqrt{4t-1}$, i.e. $t=(1+s^2)/4$, $dt=(s/2)\,ds$.
Then for integer $m\geq 0$,
\[
I_m=\int_{1/4}^{1}\frac{t^{\,m}}{\sqrt{4t-1}}\,dt
=\frac12\int_{0}^{\sqrt3}\Big(\frac{1+s^2}{4}\Big)^{m}\,ds
=\frac{\sqrt3}{2^{2m+1}}
\sum_{k=0}^{m}\binom{m}{k}\frac{3^{k}}{2k+1}.
\]
For $m<0$ let $q=-m$. We have,
\[
I_{m}=\int_{1/4}^{1}\frac{t^{m}}{\sqrt{4t-1}}\,dt
=2^{2q-1}\int_{0}^{\sqrt3}\frac{ds}{(1+s^2)^{q}}=2^{2q-1}\left[\int_0^{\sqrt 3}\frac{ds}{(1+s^2)^{q-1}}-\int_0^{\sqrt 3}\frac{s^2ds}{(1+s^2)^{q}}\right]
\]
Integrate by parts on the second term with, $u=s$ and $dv=s(1+s^2)^{-q}ds$,
\[
\int_0^{\sqrt{3}} \frac{s^2ds}{(1+s^2)^{q}}=-\frac{\sqrt 3(1+(\sqrt 3)^2)^{1-q}}{2(q-1)}+\frac{1}{2(q-1)}\cdot \int_0^{\sqrt 3}\frac{ds}{(1+s^2)^{q-1}}
\]
We get the recurrence,
\[
I_m=\frac{\sqrt3}{q-1}+\frac{2(2q-3)}{q-1}\,I_{m+1},\qquad q\ge2,
\]
with initial value
\[
I_{-1}=2\int_{0}^{\sqrt 3}\frac{du}{1+u^2}
=2\arctan(\sqrt3)=\tfrac{2\pi}{3}.
\]
Hence every $I_m$ is a $\mathbb{Q}$–linear combination of $\sqrt3$ and $\pi$.
\end{proof}

\begin{figure}
\begin{center}
\begin{tikzpicture}[scale=1.2]

\pgfmathsetmacro{\hexsize}{1}

\newcommand*\axialToXY[2]{
  \pgfmathsetmacro{\X}{1.732*(#1 + 0.5*#2)*\hexsize}
  \pgfmathsetmacro{\Y}{1.5*(#2)*\hexsize}
}

\newcommand*\Hex[2]{
  \axialToXY{#1}{#2}
  \begin{scope}[shift={(\X,\Y)}]
    \path
      \foreach \k in {0,...,5} {
        ({\hexsize*cos(30+60*\k)},{\hexsize*sin(30+60*\k)}) coordinate (P\k)
      };
    \draw[gray!55, very thin]
      (P0) -- (P1) -- (P2) -- (P3) -- (P4) -- (P5) -- cycle;
  \end{scope}
}

\newcommand*\Gval[3]{
  \axialToXY{#1}{#2}
  \node at (\X,\Y) {#3};
}

\def\Rwedge{5}
\foreach \q in {0,...,\Rwedge}{
  \pgfmathtruncatemacro{\rmax}{\Rwedge-\q}
  \foreach \r in {0,...,\rmax}{
    \Hex{\q}{\r}
  }
}

\Gval{0}{0}{$0$}
\Gval{0}{1}{$-\frac{3\sqrt 3}{2\pi}$}
\Gval{1}{0}{-1}
\Gval{1}{1}{$-\frac{9\sqrt 3}{4\pi}$}
\Gval{2}{0}{$\frac{3\sqrt 3}{\pi}-3$}

\Gval{0}{2}{$-\frac{9\sqrt 3}{4\pi}$}
\Gval{3}{0}{$\frac{27\sqrt 3}{2\pi}-9$}
\Gval{2}{1}{$1-\frac{9\sqrt 3}{2\pi}$}
\Gval{1}{2}{$-\frac{21\sqrt 3}{8\pi}$}
\Gval{0}{3}{$-\frac{27\sqrt 3}{10\pi}$}
\Gval{4}{0}{$\frac{99\sqrt 3}{2\pi}-29$}
\Gval{3}{1}{$6-\frac{111\sqrt 3}{8\pi}$}
\Gval{2}{2}{$-\frac{117\sqrt 3}{40\pi}$}
\Gval{1}{3}{$-\frac{117\sqrt 3}{40\pi}$}
\Gval{0}{4}{$-\frac{843\sqrt 3}{280\pi}$}

\Gval{5}{0}{$\frac{705\sqrt 3}{4\pi}-99$}
\Gval{4}{1}{$27-\frac{261\sqrt 3}{5\pi}$}
\Gval{3}{2}{$-\frac{27\sqrt 3}{20\pi}-1$}
\Gval{2}{3}{$-\frac{879\sqrt 3}{280\pi}$}
\Gval{1}{4}{$-\frac{1773\sqrt 3}{560\pi}$}
\Gval{0}{5}{$-\frac{909\sqrt 3}{280\pi}$}

\draw[gray!55, dashed] (0,0)--++(10,0);
\draw[gray!55, dashed] (0,0)--++(60:10);

\end{tikzpicture}
\caption{Explicit $\mathcal G_\infty(x,y)$ values in a single $60^\circ$ wedge ($x\ge0,\ y\ge0,\ x+y\le 5$). Origin is marked by $0$, and dashed lines represent $e_1$ and $e_2$ directions. Other values follow by $D_3$ symmetry.}
\label{fig:explicit_g_inf_wedge}
\end{center}
\end{figure}
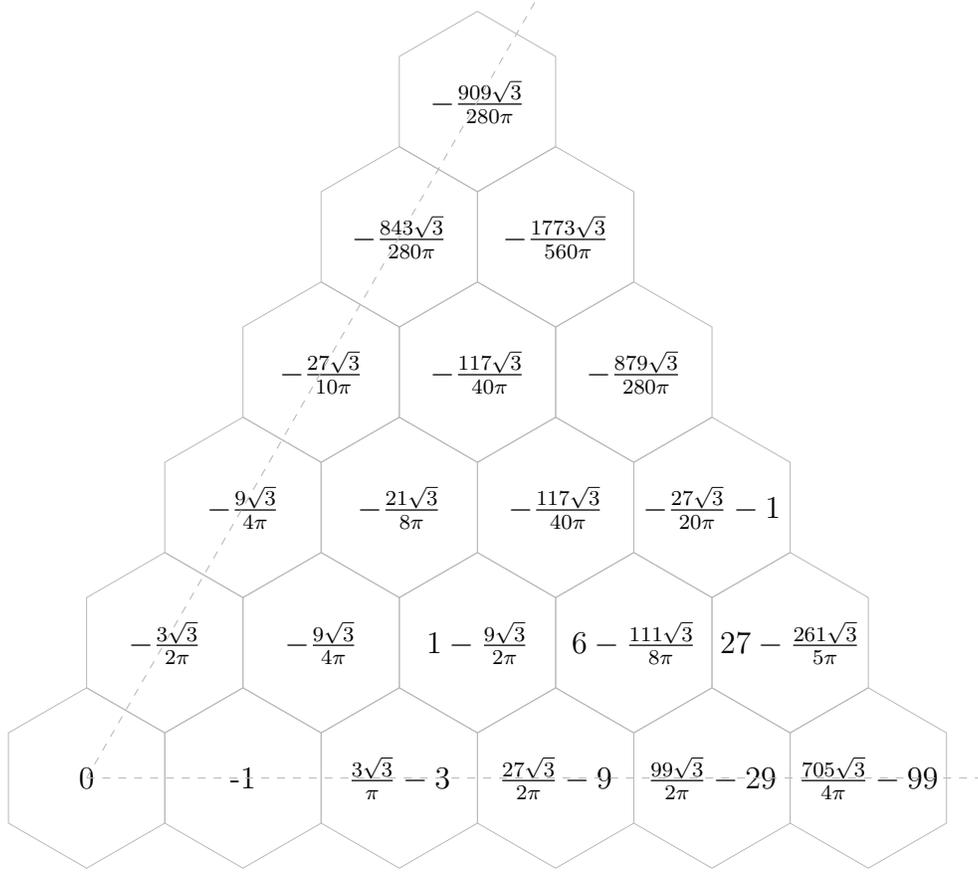

\subsection{Green's Function in Wedge}

\begin{prop}\label{prop:wedge_green_function}
Let $W \subset \mathbb{C}$ be the $60^\circ$ sector between the two rays of angles $0$ and $\pi/3$.  
Consider the dihedral group $D_3$ of order $6$ acting on the six wedges congruent to $W$.  
Then the Green's function in $W$ with Dirichlet boundary conditions is given by
\[
G_W(z, z_0) = \sum_{\gamma \in D_3} \operatorname{sgn}(\gamma) \, G_\infty(z-\gamma( z_0)),
\]
where $\operatorname{sgn}(\gamma)=1$ if $\gamma$ is a rotation by $0,2\pi/3,4\pi/3$, and $\operatorname{sgn}(\gamma)=-1$ for reflections.
\end{prop}

\begin{proof}
Consider the group $D_3$ acting on the six wedges of the same shape as $W$.  
By construction, the alternating sum
\[
\sum_{\gamma \in D_3} \operatorname{sgn}(\gamma) \, G_\infty(z-\gamma( z_0))
\]
vanishes along the boundaries of $W$ due to symmetry.  
In the interior $W \setminus \partial W$, we have 
$\Delta_z G_W(z, z_0) = \delta_{z_0}(z)$ by the definition of $G_\infty$.  
Thus $G_W$ satisfies the defining properties of the Dirichlet Green’s function on $W$.
\end{proof}
Note we equivalently have that,
\[
G_W(z, z_0) = \sum_{\gamma \in D_3} \operatorname{sgn}(\gamma) \, \mathcal G_\infty(z-\gamma( z_0)),
\]

\subsection{Point Mass boundary condition}

\begin{prop}\label{prop:point_mass_solution}
Let $a \in \partial W$ be a boundary point of the $60^\circ$ wedge $W$.  
Let $h_a(z)$ be the discrete harmonic function on $W \setminus \partial W$ with boundary values 
$h_a(z) = \delta_a(z)$ on $\partial W$.  
Then for $z \in W \setminus \partial W$ we have
\[
h_a(z) = \frac{1}{3}G_W\big(z,\, a-v_j\big),
\]
where $v_j$ is the unique step vector such that $a-v_j \in W \setminus \partial W$  
($v_3$ if $a$ lies on the $e_1$-axis, $v_2$ if $a$ lies on the $e_2$-axis).
\end{prop}

\begin{proof}
We need to show the function $h_a(z)$ satisfies $\Delta_z h_a(z) = 0$ for all $z \in W \setminus \partial W$ with the given boundary conditions. First, consider $\Delta_zh_a(z)$ for points $z\ne a-v_j$. Since the neighbors of $z$ are $z+v_1, z+v_2, z+v_3$ none of which are $a$ by the definition of $G_W$ we conclude that $\Delta_zh_a(z)=0$. When $z=a-v_j$ we get $\Delta_zh_a(z)=\frac{1}{3}\Delta_zG_w(a-v_j, a-v_j)-\frac{1}{3}\delta_a(a)=1/3-1/3=0$. Hence, the function is harmonic and satisfies the boundary conditions.
\end{proof}

\begin{ex} \label{ex:h(0,2)11}
    We compute $h_{(0,2)}(1,1)=\frac{1}{3}G_W((1,1),(1,1))$. The images of $(1,1)$ under the $D_3$ symmetries are $(1,1),(-1,2),(-2,1),(-1,-1),(1,-2),(2,-1)$. 
    Hence,
    \begin{align*}
    h_{(0,2)}(1,1)=\frac{1}{3}\left[\mathcal G_\infty(0,0)-\mathcal G_\infty(2,-1)+\mathcal G_\infty(3,0)-\mathcal G_\infty(2,2)+\mathcal G_\infty(0,3)-\mathcal G_\infty(-1,2)\right]\\
    =0+\frac{9\sqrt{3}}{4\pi}+\frac{27\sqrt 3}{2\pi}-9+\frac{117\sqrt{3}}{40\pi}-\frac{27\sqrt 3}{10\pi}+\frac{9\sqrt{3}}{4\pi}=\frac{243\sqrt{3}}{40\pi}-3\approx 0.3493
    \end{align*}
\end{ex}

\subsection{Full boundary condition}

\begin{lem} \label{lem:full_boundary_g}
Let $g(z)$ be the discrete harmonic function on $W \setminus \partial W$ with boundary conditions
\[
g(z) = 1 \ \text{on the $e_1$-ray}, 
\quad g(z) = 0 \ \text{on the $e_2$-ray}.
\]
Write $z=xe_1+ye_2$. Then
\[
g(z)=-\frac{1}{12\pi^2}\oint_{|u|=1}\oint_{|w|=1}\frac{u^{{x-1}}w^{y-1}}{\frac{1}{3}(u+u^{-1}w+w^{-1})-1}\cdot S(u, w)\,dw\,du
\]
Where 
\[
S(u, w)=
\frac{(u - w^{2})(u^{2} - w)(u w - 1)}{u w (u - 1)(u - w)(w - 1)}.
\]
\end{lem}

\begin{proof}
By linearity, $g = \sum_{m \ge 1} h_{(m,0)}$ is harmonic in $W$ and has the correct boundary values.  
From Proposition~\ref{prop:point_mass_solution} on $W\backslash \partial W$,
\[
h_{(m,0)}(z) = \frac{1}{3}G_W\big(z, me_1 + e_2\big).
\]
Summing over $m$ gives
\[
g(z) = \sum_{m=1}^\infty \frac{1}{3}G_W\big(z, me_1 + e_2\big).
\]
Using the formula for $G_W$ from Proposition~\ref{prop:wedge_green_function} and the Fourier representation
for $G_\infty$ from Lemma~\ref{lem:infinite_lattice} yields
\[
g(z) = \frac{1}{12\pi^2} \sum_{\gamma\in D_3}\mathrm{sgn}(\gamma)\sum_{m\geq 1}\int_0^{2\pi} \int_0^{2\pi} 
        \frac{e^{i(z-\gamma(me_1+e_2))\cdot (\theta, \phi)}}{\lambda(\theta, \phi)} 
        \, d\theta \, d\phi.
\]
We have, 
\[
\sum_{m\geq 1}e^{-im(\gamma e_1)\cdot (\theta, \phi)}=\frac{e^{-i(\gamma e_1)\cdot (\theta, \phi)}}{1-e^{-i(\gamma e_1)\cdot (\theta, \phi)}}
\]
The possible images of $\gamma (e_1+e_2)$ are $(1,1),(-2,1),(1,-2)$, $(-1,2),(2,-1),(-1,-1)$ and the corresponding images of $\gamma(e_1)$ are $(1,0),(-1,1),(0,-1),(-1,1),(1,0),(0,-1)$. We let $u=e^{i\theta}$ and $w=e^{i\phi}$. Then,
\begin{align*}
\sum_{\gamma\in D_3}\mathrm{sgn}(\gamma)\sum_{m\geq 1}\frac{e^{-i\gamma(e_1+e_2))\cdot (\theta, \phi)}}{1-e^{-i\gamma e_1\cdot (\theta, \phi)}}= \frac{u^{-1}w^{-1}}{1-u^{-1}}+\frac{u^2w^{-1}}{1-uw^{-1}}+\frac{u^{-1}w^{2}}{1-w}-\frac{uw^{-2}}{1-uw^{-1}}-\frac{u^{-2}w}{1-u^{-1}}-\frac{uw}{1-w}
\end{align*}
Thus,
\[
g(z)=-\frac{1}{12\pi^2}\oint_{|u|=1}\oint_{|w|=1}\frac{u^{x-1}w^{y-1}}{\frac{1}{3}(u+u^{-1}w+w^{-1})-1}\cdot S(u, w)\,dw\, du
\]
Where, \[
S(u, w)= \frac{u^{-1}w^{-1}}{1-u^{-1}}+\frac{u^2w^{-1}}{1-uw^{-1}}+\frac{u^{-1}w^{2}}{1-w}-\frac{uw^{-2}}{1-uw^{-1}}-\frac{u^{-2}w}{1-u^{-1}}-\frac{uw}{1-w}
\]
We can factor,
\[
S(u, w)=
\frac{(u - w^{2})(u^{2} - w)(u w - 1)}{u w (u - 1)(u - w)(w - 1)}.
\]
\end{proof}

\section{Faces at Distance \texorpdfstring{$d$}{d} from the Boundary} \label{sec:faces_m}

In this section, we study the distribution of face sizes in uniformly random reduced $\mathfrak{sl}_3$ webs that lie at distance at least $d$ from the boundary. Our approach is to work with face diagrams which are abstract local configurations of arcs surrounding a single interior face. By analyzing the probabilities of these local configurations, we can compare the asymptotic frequencies of faces of different sizes. In particular, we define a face diagram extension operation that transforms a face diagram of size $2k$ into one of size $2k+2$, and we compare the probabilities of these two configurations appearing at depth at least $d$.

The key idea is to reduce the problem to analyzing discrete harmonic functions in a truncated quadrant, denoted $Q_d \subset \mathbb{Z}^2$. This region consists of all lattice points in the first quadrant lying outside the triangle bounded by the coordinate axes and the line $x + y = d$. The region $Q_d$ models the subdiagram beneath a face that is at distance at least $d$ from the boundary. The condition that the distance to the boundary is at least $d$, corresponds to the condition $x + y \geq d$. By solving discrete Dirichlet problems on $Q_d$, we obtain hitting probabilities that encode the likelihood that a particular face diagram is at least distance $d$ from the boundary. Using these estimates, we prove that the ratio of the probabilities of extended versus original face diagrams is bounded above by $O(1/d^2)$, and thus that the fraction of interior faces of size greater than 6 converges to zero as $d \to \infty$. This also implies that the number of faces of size $6+2k$ and a distance at least $d$ from the boundary is of the order $O(1/d^{2k})$.

\begin{defn}
Let $F$ be an interior face of a reduced web.  
The \emph{depth} (or \emph{distance to the boundary}) of $F$ is the minimal number of edges in the dual graph of the web that must be traversed to reach a face incident to the boundary.  
Equivalently, it is the length (in dual edges) of a shortest path in the dual graph from $F$ to any boundary face.  
\end{defn}

\begin{defn}
\label{defn:hitting_probs}
    Fix an integer $d \geq 0$. Let $Q_d$ denote the region in the upper-right quadrant of $\mathbb{Z}^2$ with the triangle bounded by the vertices $(0,0)$, $(0,d-1)$, and $(d-1,0)$ removed. That is,
    \[
    Q_d := \left\{(x,y)\in\mathbb{Z}_{\geq 0}^2 \,:\, x+y \geq d \right\}.
    \]
    Define the discrete Laplacian operator on functions $f : \mathbb{Z}^2 \to \mathbb{R}$ by
    \[
    \Delta f(x, y) = \frac{1}{3} \left[f(x+1, y) + f(x-1, y+1) + f(x, y-1) - 3f(x, y)\right].
    \]
    For a boundary point $a \in \partial Q_d$, define the function $h_a(\cdot;d):\mathbb{Z}^2 \to \mathbb{R}$ to be the unique solution to the discrete Dirichlet problem
    \[
    \Delta h_a(x, y;d) = 0 \quad \text{for } (x,y) \in Q_d, \qquad h_a(x,y;d) = \delta_a(x,y) \text{ on } \partial Q_d.
    \]
    In other words, $h_a(\cdot;d)$ is the discrete harmonic function on $Q_d$ with a unit source at $a$ and vanishing boundary conditions elsewhere. See Figure \ref{fig:domain_qd}.
\end{defn}

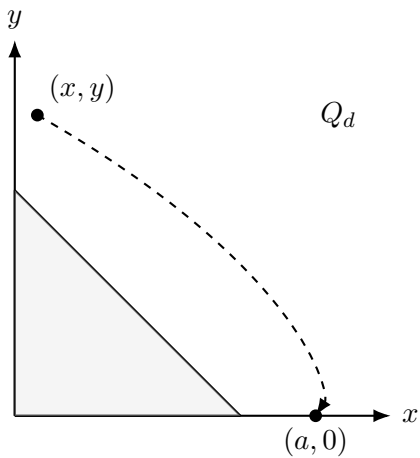
\begin{figure}[ht]
\begin{center}
\begin{tikzpicture}[scale=1, thick, 
    every node/.style={font=\small},
    every path/.style={>=Latex}]
    
  \draw[->] (0,0) -- (5,0) node[right] {$x$};
  \draw[->] (0,0) -- (0,5) node[above] {$y$};
  
  \draw[thick] (0,0) -- (3,0);
  \draw[thick] (0,0) -- (0,3);
  \draw[thick] (0,3) -- (3,0);
  
  \filldraw[black] (0.3,4) circle (2pt) node[above right] {$(x,y)$};
  
  \filldraw[black] (4,0) circle (2pt) node[below] {$(a,0)$};
  
  \draw[dashed,->] (0.3,4) .. controls (3.8,2) and (4.2,0.3) .. (4,0);

  \node at (4.3,4) {$Q_d$};

  \fill[gray!20, opacity=0.4] (0,0) -- (0,3) -- (3,0) -- cycle;

\end{tikzpicture}
\end{center}
\caption{Domain $Q_d$}
\label{fig:domain_qd}
\end{figure}

\begin{defn}
A \emph{face diagram} is an abstract local configuration in an $m$-diagram consisting of a single interior face together with the surrounding arcs that form its immediate boundary. Specifically, a face of size $2k$ is bounded by $2k-2$ arc segments alternating in color; the face diagram records only these boundary arc segments in order. 

The \emph{face diagram extension operation} is a local transformation that produces a new face diagram corresponding to a face of size $2k+2$ from one of size $2k$. Identify the first crossing point (from left to right) beneath the face at which two boundary arcs intersect (the \emph{green point} in Figure~\ref{fig:face-extension}. Uncross these two arcs and insert one new red arc and one new blue arc between the separated strands. This adds exactly two more arc segments to the boundary of the face, increasing its size by two.

At this stage, a face diagram should be viewed simply as a possible configuration of arc segments around a certain face, without specifying which reduced web it belongs to. The probability calculations in later sections determine how often such configurations occur in uniformly random reduced webs. Considering all possible face diagrams will determine the probabilities of faces of certain sizes in a reduced web.
\end{defn}

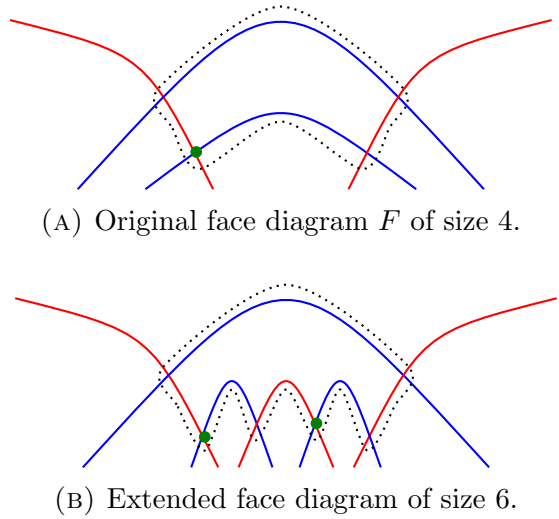
\begin{figure}[ht]
\centering

\begin{subfigure}[b]{0.48\textwidth}
\centering
\begin{tikzpicture}[scale=0.9, thick]
  \draw[red] (0,2.5) .. controls (2,2) .. (3,0);
  \draw[blue] (1,0) .. controls (4,3.3) .. (7,0);
  \draw[blue] (2,0) .. controls (4,1.5) .. (6,0);
  \draw[red] (5,0) .. controls (6,2) .. (8,2.5);

  \filldraw[green!50!black] (2.75,0.55) circle (2pt);
    \draw[dotted, thick] plot[smooth cycle] coordinates { (2.2,1.5) (2.4,0.9) (2.8,0.3) (4,1) (5.2,0.3) (5.6,0.9) (5.8,1.5) (4,2.7)};
\end{tikzpicture}
\caption{Original face diagram $F$ of size $4$.}
\label{fig:face-original}
\end{subfigure}
\hfill

\begin{subfigure}[b]{0.48\textwidth}
\centering
\begin{tikzpicture}[scale=0.9, thick]
  \draw[red] (0,2.5) .. controls (2,2) .. (3,0);
  \draw[blue] (1,0) .. controls (4,3.3) .. (7,0);
  \draw[red] (5,0) .. controls (6,2) .. (8,2.5);

  \draw[blue] (2.6,0) .. controls (3.2,1.7) .. (3.8,0);
  \draw[red] (3.3,0) .. controls (4,1.7) .. (4.7,0);
  \draw[blue] (4.2,0) .. controls (4.8,1.7) .. (5.4,0);
  
  \filldraw[green!50!black] (2.8,0.45) circle (2pt);
  \filldraw[green!50!black] (4.45,0.65) circle (2pt);
    \draw[dotted, thick] plot[smooth cycle] coordinates { (2.2,1.5) (2.4,0.9) (2.8,0.25) (3.2,1.15) (3.55,0.4) (4,1.15) (4.45,0.4) (4.8,1.15) (5.2,0.25) (5.6,0.9) (5.8,1.5) (4,2.7)};
\end{tikzpicture}
\caption{Extended face diagram of size $6$.}
\label{fig:face-extended}
\end{subfigure}

\caption{An extended face diagram is obtained from the original face diagram (\subref{fig:face-original}) by resolving a crossing (green dot) and inserting one red and one blue arc (\subref{fig:face-extended}), which increases the face size by two. The dotted curve indicates the boundary of the face diagram.}
\label{fig:face-extension}
\end{figure}
\FloatBarrier

\begin{lem} \label{lem:face_four_diagram}
Let $F$ be an interior face of an $m$-diagram
whose boundary is formed by segments of exactly four arcs.  There are exactly four possible local face diagrams
around $F$, which are shown in Figure \ref{fig:face-six-diagrams}. After resolution to a web, each such quadrilateral becomes a face of size $6$.
\end{lem}

\begin{figure}[ht]
\centering

\begin{subfigure}[b]{0.24\textwidth}
\centering
\begin{tikzpicture}[scale=0.5, thick]
  \draw[red] (0,2.5) .. controls (2,2) .. (3,0);
  \draw[blue] (1,0) .. controls (4,3.3) .. (7,0);
  \draw[blue] (2,0) .. controls (4,1.5) .. (6,0);
  \draw[red] (5,0) .. controls (6,2) .. (8,2.5);

  \filldraw[green!50!black] (2.75,0.55) circle (2pt);
    \draw[dotted, thick] plot[smooth cycle] coordinates { (2.2,1.5) (2.4,0.9) (2.8,0.3) (4,1) (5.2,0.3) (5.6,0.9) (5.8,1.5) (4,2.7)};
\end{tikzpicture}
\end{subfigure}
\hfill

\begin{subfigure}[b]{0.24\textwidth}
\centering
\begin{tikzpicture}[scale=0.5, thick]
  \draw[blue] (0,2.5) .. controls (2,2) .. (3,0);
  \draw[red] (1,0) .. controls (4,3.3) .. (7,0);
  \draw[red] (2,0) .. controls (4,1.5) .. (6,0);
  \draw[blue] (5,0) .. controls (6,2) .. (8,2.5);

  \filldraw[green!50!black] (2.75,0.55) circle (2pt);
    \draw[dotted, thick] plot[smooth cycle] coordinates { (2.2,1.5) (2.4,0.9) (2.8,0.3) (4,1) (5.2,0.3) (5.6,0.9) (5.8,1.5) (4,2.7)};
\end{tikzpicture}
\end{subfigure}
\begin{subfigure}[b]{0.24\textwidth}
\centering
\begin{tikzpicture}[scale=0.5, thick]
  \draw[blue] (1,1.5) .. controls (2,1.4) .. (4,0); 
  \draw[red] (5,1.5) .. controls (4,1.4) .. (2,0); 
  \draw[red] (1,0) .. controls (3,2) .. (5,2.5); 
  \draw[blue] (5,0) .. controls (3,2) .. (1,2.5); 

  \filldraw[green!50!black] (3,0.7) circle (2pt); 
\draw[dotted, thick] plot[smooth cycle] coordinates { (2.05,1.2) (3, 0.6) (3.95,1.2) (3, 2.0)};
\end{tikzpicture}
\end{subfigure}
\begin{subfigure}[b]{0.24\textwidth}
\centering
\begin{tikzpicture}[scale=0.5, thick]

  \draw[red] (1,1.5) .. controls (2,1.4) .. (4,0); 
  \draw[blue] (5,1.5) .. controls (4,1.4) .. (2,0); 
  \draw[blue] (1,0) .. controls (3,2) .. (5,2.5); 
  \draw[red] (5,0) .. controls (3,2) .. (1,2.5); 

  \filldraw[green!50!black] (3,0.7) circle (2pt);
\draw[dotted, thick] plot[smooth cycle] coordinates { (2.05,1.2) (3, 0.6) (3.95,1.2) (3, 2.0)};
\end{tikzpicture}
\end{subfigure}
\caption{Possible face diagrams with $4$ arcs (corresponding to a face of size $6$ in the reduced web)}
\label{fig:face-six-diagrams}
\end{figure}
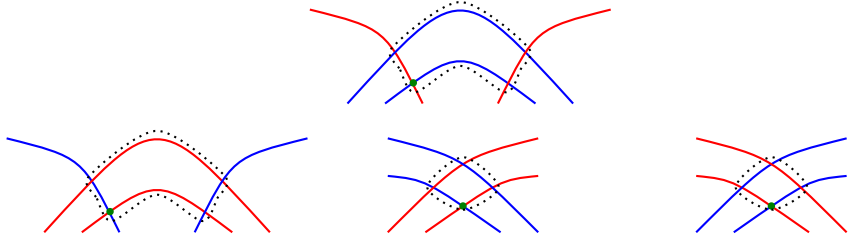

\begin{proof}
    By $(1)$ of Lemma \ref{lem:face-size-drop-by-two} the leftmost and rightmost point of the $F$ are connected by either one or two arc segments above the face. This gives us $4$ total options since we can pick either one or two segments and we also have a choice of red/blue. The remaining arc segments bounding $F$ from below must alternate between red and blue and are uniquely determined once the top arc segments are chosen. Hence we obtain the four diagrams in Figure \ref{fig:face-six-diagrams}.
\end{proof}

\begin{prop} \label{prop:face_diagrams}
    Any face diagram $F$ of size $2k+4$ (that is $2k+4$ arc segments) with $k\geq 0$, can be obtained by repeating the diagram extension operation to one of the four face diagrams of size $4$ (Figure \ref{fig:face-six-diagrams}) $k$ times.
\end{prop}

\begin{proof}
    By $(1)$ of Lemma \ref{lem:face-size-drop-by-two} the leftmost and rightmost intersection point of the $F$ are connected by either one or two arc segments above the face. Below the face is a sequence of alternating red, blue arc segments. Find the face diagram of size $4$ that has the same arc segments above its face as $F$. Then repeat the face diagram extension operation $k$-times.
\end{proof}

We compare the total number of occurrences of a fixed face diagram and of its
extended version across all reduced webs associated to $3\times n$ tableaux. Let $F$ be a face diagram of a face of size $2k$. Consider the intersection point immediately to the left of the face of $F$, marked by $p$ in Figure~\ref{fig:crossing_face_depth}. Following the two arcs that intersect at $p$, we trace them toward the boundary until they terminate at points labeled $s$ and $e$.

We condition on the number of arcs that cross these segments. Specifically, let $F_a$ denote the number of arcs that cross the segment from $s$ to $p$, and let $F_b$ denote the number of arcs that cross the segment from $p$ to $e$, minus one.

Next, we consider the corresponding extended face diagram, obtained by resolving the crossing at the green point and inserting a red and a blue arc, as shown in Figure~\ref{fig:crossing_face_depth_2}. In the extended face diagram the point $e$ is moved to where the last inserted arc touches the boundary as shown. Note the values of $F_a$ and $F_b$ remain unchanged.

The main observation is that this modification affects only the portion of the diagram between $s$ and $e$. Outside of this region, the original and extended diagrams are isomorphic and contribute identically to the probability that the face lies at distance at least $m$ from the boundary. Therefore, to compare the probabilities associated to the original and extended face diagrams, it suffices to analyze the subdiagram between $s$ and $e$, conditioned on $F_a$ and $F_b$. Let $ F_{se} $ denote the region of the $ m $-diagram between the steps $ s $ and $ e $ that lies beneath the face $ F $. Let $F_{es}$ denote the remaining region of the $m$-diagram that does not include $F$ or $F_{se}$.

\begin{figure}[ht]
\centering
\begin{minipage}[b]{0.48\textwidth}
\centering
\begin{tikzpicture}[scale=0.9, thick]

  \draw (0,0) -- (8,0);

  \draw[red] (2,1.6) .. controls (2.1,1.5) .. (3,0);
  \draw[red] (1.9,1.4) .. controls (2,1.3) .. (2.2,1);
  \draw[red] (1.8,1.2) .. controls (1.9,1.1) .. (2.1,0.8);
  \draw[red] (1.7,1) .. controls (1.8,0.9) .. (2,0.6);

  \draw[blue] (2,0) .. controls (4,1.5) .. (6,0);
  \draw[blue] (1,0) .. controls (4,3.3) .. (7,0);
  \draw[blue] (2.6,0.25) .. controls (2.75,0.35) .. (3,0.45);
  \draw[blue] (3,0) .. controls (3.2,0.25) .. (3.3,0.35);

  \draw[red] (5,0) .. controls (5.9,1.5) .. (6,1.6);

  \node at (1,-0.2) {$s$};
  \node at (3,-0.2) {$e$};
  \filldraw[black] (2.22,1.32) circle (2pt);
  \node at (2.25,1.7) {$p$};

  \filldraw[green!50!black] (2.7,0.5) circle (2pt);

  \fill[gray!20, opacity=0.5] (1,0) -- (2.22,1.3) -- (3,0) -- cycle;
\end{tikzpicture}
\caption{Face $F$ diagram of a face of size $4$, with $F_a = 4$ red arcs crossing over the left endpoint $s$ and $F_b = 3$ blue arcs crossing over the right endpoint $e$. $F_{se}$ shaded in grey.}
\label{fig:crossing_face_depth}
\end{minipage}
\hfill
\begin{minipage}[b]{0.48\textwidth}
\centering
\begin{tikzpicture}[scale=0.9, thick]
  \draw (0,0) -- (8,0);

  \draw[red] (2,1.6) .. controls (2.1,1.5) .. (3,0);
  \draw[red] (1.9,1.4) .. controls (2,1.3) .. (2.2,1);
  \draw[red] (1.8,1.2) .. controls (1.9,1.1) .. (2.1,0.8);
  \draw[red] (1.7,1) .. controls (1.8,0.9) .. (2,0.6);
  \draw[blue] (1,0) .. controls (4,3.3) .. (7,0);
  \draw[red] (5,0) .. controls (5.9,1.5) .. (6,1.6);

  \draw[blue] (2.6,0) .. controls (3.2,1.7) .. (3.8,0);
  \draw[red] (3.3,0) .. controls (4,1.7) .. (4.7,0);
  \draw[blue] (4.2,0) .. controls (4.8,1.7) .. (5.4,0);
  \draw[blue] (4.45,0.25) .. controls (4.55,0.45) .. (4.65,0.55);
  \draw[blue] (4.7,0) .. controls (4.8,0.35) .. (4.9,0.45);

  \node at (1,-0.2) {$s$};
  \node at (4.7,-0.2) {$e$};
  \node at (2.25,1.7) {$p$};
  \filldraw[black] (2.22,1.32) circle (2pt);

  \filldraw[green!50!black] (2.73,0.45) circle (2pt);
  \filldraw[green!50!black] (4.45,0.65) circle (2pt);
  \node at (4.48,1.2) {$p'$};

  \fill[gray!20, opacity=0.5] (1,0) -- (2.22,1.3) -- (3,0) -- cycle;

  \fill[gray!30, opacity=0.5]
    (2.6, 0)
    .. controls (3.2,1.7) ..
    (3.8, 0) -- (2.6,0) -- (3.8, 0) -- cycle;
     \fill[gray!30, opacity=0.5]
    (3.3, 0)
    .. controls (4,1.7) ..
    (4.7, 0) -- (3.3,0) -- (4.7, 0) -- cycle;
\end{tikzpicture}
\caption{Extended face diagram of a face of size $6$, obtained by uncrossing the arcs at the green point and inserting one red and one blue arc. The values of $F_a$ and $F_b$ remain the same. $F_{se}$ shaded in grey.}
\label{fig:crossing_face_depth_2}
\end{minipage}
\end{figure}

\begin{lem} \label{lem:shortest_dual_regions}
    The shortest path from a face $ F $ to the boundary is equal in length to the shortest path from $ F $ to the boundary that lies entirely within either the region $ F_{se} $ or the region $ F_{es} $.
\end{lem}

\begin{proof}
    Suppose for contradiction that there exists a shortest path $ \gamma $ in the dual graph from $ F $ to a boundary face that passes through both $ F_{se} $ and $ F_{es} $, and that no path of equal or shorter length exists entirely within either region.

    Let $ e_0 $ denote the last dual edge in $ \gamma $ that crosses between the two regions $ F_{se} $ and $ F_{es} $. Let $ f_1 $ and $ f_2 $ be the faces that correspond to the vertices of $ e_0 $, with $ f_1 \in F_{se} $ and $ f_2 \in F_{es} $, and suppose the path moves from $ f_1 $ to $ f_2 $. Let $ \gamma_1 $ be the subpath of $ \gamma $ from $ F $ to $ f_2 $.

    Now consider an alternate path $ \gamma_2 $ from $ F $ to $ f_2 $ that travels entirely within $ F_{se} $, following parallel to the arc segments $ (p, s) $ or $ (p, e) $ (or $ (p', e) $ in the extended face setting). Assuming we exit $F$ in the correct direction such a path can be chosen to remain in one region.

    Furthermore, because the arcs of the same color crossing $ (p, s) $ or $ (p, e) $ (or $ (p', e) $ in the extended face setting) are nested, the alternate path $ \gamma_2 $ has length at most equal to that of $ \gamma_1 $. Replacing $ \gamma_1 $ with $ \gamma_2 $ yields a path from $ F $ to the boundary of the same or shorter length that lies entirely within $ F_{se} $ or $ F_{es} $, contradicting the assumption.

    Therefore, any shortest path from $ F $ to the boundary may be taken to lie entirely within one of the two regions.
\end{proof}

We now formalize the comparison between the probabilities that a face in a given face diagram 
and its corresponding extended face diagram lie at distance at least $m$ from the boundary. 
Recall from Proposition~\ref{prop:face_diagrams} that there are four possible face diagrams of size $2k$ 
that can serve as the base configuration.  
Throughout, we fix one such base diagram and condition on the crossing data 
$(F_a, F_b)$, where $F_a$ (resp.\ $F_b$) denotes the number of arcs crossing the segment from 
$p$ to $s$ (resp.\ from $p$ to $e$) in the $m$-diagram.  

Let $T_1(F_a, F_b, n, d, k)$ be the number of occurrences of the chosen face diagram of size $2k$ 
with crossing data $(F_a, F_b)$ among all reduced webs in $W_n$, such that the interior face is at 
distance at least $d$ from the boundary.  
Similarly, let $T_2(F_a, F_b, n, d, k)$ denote the number of occurrences of the extended face diagram 
(obtained from the original via the arc-insertion operation) with the same crossing data $(F_a, F_b)$, 
and such that the interior face is also at distance at least $d$ from the boundary.

We define the normalized limiting densities:
\[
L_d^1 := \lim_{n \to \infty} \frac{T_1(F_a, F_b, n, d, k)}{n \cdot |W_n|}, \quad
L_d^2 := \lim_{n \to \infty} \frac{T_2(F_a, F_b, n, d, k)}{n \cdot |W_n|}.
\]

\begin{thm}
\label{thm:extension_probability_ratio}
Fix $m \geq 0$ and crossing numbers $(F_a, F_b)$. Then the ratio of limiting probabilities $L_d^2 / L_d^1$ is given by a discrete hitting probability expression involving the functions $h_a(x, y; d)$ defined in Definition \ref{defn:hitting_probs}. The exact formula depends on the color of the arc beginning at $s$:

\begin{itemize}
    \item If the arc at $s$ is blue, then
    \[
    \frac{L_d^2}{L_d^1} =
    \sum_{t, s = d+1}^\infty
    \frac{h_{(0,t)}(F_a,1;d) \cdot h_{(s,0)}(1,t-1;d) \cdot h_{(0,F_b+1)}(s-1,1;d)}
         {h_{(0,F_b+1)}(F_a,1;d)}.
    \]

    \item If the arc at $s$ is red, then
    \[
    \frac{L_d^2}{L_d^1} =
    \sum_{t, s = d+1}^\infty
    \frac{h_{(t,0)}(1,F_a;d) \cdot h_{(0,s)}(t-1,1;d) \cdot h_{(F_b+1,0)}(1,s-1;d)}
         {h_{(F_b+1,0)}(1,F_a;d)}.
    \]
\end{itemize}
Note that this ratio does not depend on the face diagram that is chosen except for the color of the beginning arc.
\end{thm}

\begin{proof}
Fix the crossing configuration $(F_a, F_b)$ as in the theorem, and consider the arc segment between the endpoints $s$ and $e$ that bounds the face $F$. As in Figure \ref{fig:crossing_face_depth} and Figure \ref{fig:crossing_face_depth_2}, the structure of the diagram outside the interval $[s,e]$ remains unchanged under the face diagram extension operation. 

The probability that the shortest path in the outer region $F_{es}$ is at least $d$ is the same for both diagrams and is independent of the probability that the shortest path in the region $F_{se}$ is at least $d$. Therefore, those probabilities will cancel in the ratio.

We consider the sequence along the segment between $s$ and $e$. We follow a similar argument to Propositions \ref{prop:crossing_probabilities_red} and \ref{prop:crossing_probabilities_blue}. The only difference is we restrict the domain to $Q_d$. Observe that if any point in time between $s$ and $e$ there are less than $d$ arcs open then we can draw a path through those arcs that reaches the boundary in less than $d$ dual edges.

Now consider the crossing probabilities between $s$ and $e$. The expression for these probabilities is given by products of discrete harmonic functions $h_a(x,y; d)$, where $h_a(x,y; d)$ represents the probability that a path starting at $(x, y)$ first reaches the boundary at the point $a$, while remaining entirely inside the region $Q_d$. Now we follow a similar argument to \ref{thm:face_probabilities}.

\smallskip

Case 1: Arc at $s$ is blue. The original configuration involves three path segments:

\begin{enumerate}
    \item A blue path from the starting point $(F_a,1)$ to some intermediate point $(0,t)$,
    \item A red path from $(1,t{-}1)$ to $(s,0)$,
    \item A blue path from $(s{-}1,1)$ to $(0,F_b+1)$.
\end{enumerate}

The numerator in the expression corresponds to summing over all such intermediate points $(t, s)$ with $t, s \geq d+1$ to ensure that the face stays at distance at least $d$ from the boundary at every step. The full probability that the extended face lies at distance at least $d$ is then given by
\[
L^2_d \propto \sum_{t, s = d+1}^\infty h_{(0,t)}(F_a,1;d) \cdot h_{(s,0)}(1,t{-}1;d) \cdot h_{(0,F_b+1)}(s{-}1,1;d),
\]
while the probability that the original face lies at distance at least $d$ is simply
\[
L^1_d \propto h_{(0,F_b+1)}(F_a,1;d).
\]
Taking the ratio gives the formula in the theorem.

\smallskip

Case 2: Arc at $s$ is red. A symmetric argument applies, with red and blue roles swapped. The corresponding expression is:
\[
\frac{L^2_d}{L^1_d} =
\sum_{t, s = d+1}^\infty
\frac{h_{(t,0)}(1,F_a;d) \cdot h_{(0,s)}(t{-}1,1;d) \cdot h_{(F_b+1,0)}(1,s{-}1;d)}
     {h_{(F_b+1,0)}(1,F_a;d)}.
\]

\smallskip

In both cases, the denominator represents the probability that the original configuration stays in $Q_d$ and reaches its terminal point, while the numerator accounts for the three-part path of the extended configuration, still constrained to remain within $Q_d$ throughout. Since all other contributions to the face's distance from the boundary occur outside $[s,e]$ and are identical between the two diagram, this completes the comparison.
\end{proof}

\begin{defn}
Let $\tilde Q_d \subset \mathbb{C}$ denote the $60^\circ$ wedge with an equilateral triangle of side length $d$ removed from the origin. Define the shear map
\[
S: \mathbb{C} \to \mathbb{C}, \quad S(1) = 1, \quad S(i) = \frac{1}{2} + \frac{\sqrt{3}}{2}i.
\]
Then $S$ sends the domain $Q_d$ (the first quadrant with a triangle removed) to $\tilde Q_d$.

We will perform calculations in $\tilde Q_d$ because under this shear, the steps of the lattice walk $1, 1 - i, -1$ are mapped to:
\[
1, \quad -\frac{1}{2} + \frac{\sqrt{3}}{2}i, \quad -\frac{1}{2}-\frac{\sqrt{3}}{2}i,
\]
As a result, the scaling limit of the random walk in $\tilde Q_d$ is a Brownian motion with covariance matrix
\[
\begin{pmatrix}
1 & 0 \\
0 & 1
\end{pmatrix}
\quad \text{(up to scaling by a constant)}.
\]
\end{defn}

\begin{prop}
\label{prop:Poisson_kernel}
The wedge with a triangle removed can be obtained from the upper half--plane by a Schwarz--Christoffel map 
with prevertices at $\pm 1$, producing interior angles of $2\pi/3$ at those vertices.
Let $\phi: \mathbb{H} \to \tilde Q_d$ be the conformal map given by
\[
\phi(\zeta) = A+B\int^\zeta \frac{C}{(w+1)^{1/3}(w-1)^{1/3}} \, dw,
\]
where the constant $C$ is chosen so that the interval $[-1,1]$ on the real line maps to a segment of length $d$ in $\tilde Q_d$. $A$ and $B$ are some constants with $|B|=1$. Let $z  \in \mathbb{H}, \operatorname{Im}(z)>0$ and $t \in \mathbb{R} \setminus (-1,1)$, so that  $\phi(t) \in \partial Q_d$ lies on one of the wedge rays. Then the Poisson kernel in $\tilde Q_d$ satisfies:
\[
P_{\tilde Q_d}(\phi(z), \phi(t)) = \frac{1}{\pi} \cdot \frac{\operatorname{Im}(z)}{|z - t|^2} \cdot |\phi'(t)|^{-1}.
\]
\end{prop}

\begin{proof}
The conformal map $\phi$ is obtained from the Schwarz--Christoffel formula mapping the upper half-plane $\mathbb{H}$ to the region $\tilde Q_d$. The integrand
\[
\frac{1}{(w+1)^{1/3}(w-1)^{1/3}}
\]
produces internal angles of $2\pi/3$ at $w = \pm 1$, matching the corner angles of the removed equilateral triangle.

To determine the constant $C$, we compute
\[
\int_{-1}^1 (1 + w)^{-1/3}(1 - w)^{-1/3} \, dw.
\]
Substituting $w = 2x - 1$, we obtain
\[
2^{1/3} \int_0^1 x^{-1/3}(1 - x)^{-1/3} \, dx = 2^{1/3} B(2/3, 2/3) = 2^{1/3} \cdot \frac{\Gamma(2/3)^2}{\Gamma(4/3)}.
\]
Setting $C = d \cdot 2^{-1/3}\cdot B(2/3, 2/3)^{-1}$ ensures that the image of $[-1,1]$ has length $d$, matching the corresponding side length in $\tilde Q_d$.

The Poisson kernel in $\mathbb{H}$ is given by
\[
P_{\mathbb{H}}(z, t) = \frac{1}{\pi} \cdot \frac{\operatorname{Im}(z)}{|z - t|^2}.
\]
By conformal invariance, the Poisson kernel transforms under $ \phi $ as
\[
P_{\tilde Q_d}(\phi(z), \phi(t)) = P_{\mathbb{H}}(z, t) \cdot |\phi'(t)|^{-1} = \frac{1}{\pi} \cdot \frac{\operatorname{Im}(z)}{|z - t|^2} \cdot |\phi'(t)|^{-1}.
\]
Finally, we observe that the domain $ \tilde Q_d $ obtained from the Schwarz--Christoffel mapping lies in a rotated and translated position in the complex plane. However, we may apply a composition of a rotation and translation so that $ \tilde Q_d $ coincides with the image of the shear map $ S $ applied to $ Q_d $. The Poisson kernel is preserved under rigid motions, so this does not affect the formula for $ P_{\tilde Q_d} $.
\end{proof}

\begin{prop}
\label{prop:ratio_decay}
Take any $d\geq 1$ possibly with $d \to \infty$ as $n \to \infty$, but $d = o(n)$. Then for any face diagram,
\[
\frac{L^2_d}{L^1_d} \lesssim \frac{1}{d^2}.
\]
\end{prop}

\begin{proof}
Recall that if the arc at $s$ is blue, then
\[
\frac{L_d^2}{L_d^1} =
\sum_{t, s = d+1}^\infty
\frac{h_{(0,t)}(F_a,1;d) \cdot h_{(s,0)}(1,t-1;d) \cdot h_{(0,F_b+1)}(s-1,1;d)}
     {h_{(0,F_b+1)}(F_a,1;d)}.
\]
The red case is analogous. Let $\phi$ be the conformal map from Proposition~\ref{prop:Poisson_kernel}, which maps the real axis to the boundary of the wedge $ \tilde Q_d $, sending $ (-\infty, -1] $ and $ [1, \infty) $ to the wedge rays $ R_1 $ and $ R_2 $, respectively.

Let $ e_1 = 1 $ and $ e_2 = -\tfrac{1}{2} + \tfrac{\sqrt{3}}{2}i $ be a basis for $ \mathbb{C} $, so that $ R_1 $ is parallel to $ e_1 $ and $ R_2 $ to $ e_2 $. Let $ z \in \mathbb{H} $ satisfy $ \phi(z) = F_a e_1 + e_2 \in R_1 $, and let $ w \in \mathbb{H} $ with $ \phi(w) = (F_b+1)e_2 \in R_2 $. As $ d \to \infty $, the discrete harmonic functions $ h $ converge to the Poisson kernel, so by Proposition~\ref{prop:Poisson_kernel},
\[
L_d^1 \sim P_{\tilde Q_d}(\phi(z), \phi(w)) = \frac{1}{\pi} \cdot \frac{\operatorname{Im}(z)}{|z - w|^2} \cdot |\phi'(w)|^{-1}.
\]

Now consider the three-step path from $ F_a e_1 + e_2 $ to $ (F_b+1)e_2 \in R_2 $, passing through intermediate points $ t e_2 \in R_2 $ and $ s e_1 \in R_1 $. The expected contribution is
\[
L_d^2 \sim \iint_{t \in R_2,\ s \in R_1} P_{\tilde Q_d}(\phi(z), t e_2) \cdot P_{\tilde Q_d}(e_1 + (t-1)e_2, s e_1) \cdot P_{\tilde Q_d}(s(e_1 - 1) + e_2, \phi(w)) \, ds \, dt.
\]

Let $t_1=\phi^{-1}(te_2)$, $t_2=\phi^{-1}(e_1 + (t-1)e_2)$, $s_1=\phi^{-1}(se_2)$, and $s_2=\phi^{-1}(s(e_1-1)+e_2)$. Then we have,
\begin{align*}
\frac{L_d^2}{L_d^1}
&\sim |z - w|^2 \cdot \iint_{t \in R_2,\ s \in R_1} 
\frac{|\phi'(t_1)|^{-1}}{|z - t_1|^2}
\cdot \frac{\operatorname{Im}(t_2) \cdot |\phi'(s_1)|^{-1}}{|t_2 - s_1|^2} \cdot\frac{\operatorname{Im}(s_2)}{|s_2 - w|^2} \, ds \, dt.
\end{align*}

Then on RHS we can rewrite the integral as 
\begin{align*}
\frac{L_d^2}{L_d^1}
&\sim |z - w|^2 \cdot \iint_{t \in R_2,\ s \in R_1} 
P_{\tilde Q_d}(e_1 + (t-1)e_2, s e_1)
\cdot P_{\tilde Q_d}(s(e_1-1)+e_2, te_2)\cdot\frac{|s_2-t_1|^2}{|z-t_1|^2|s_2 - w|^2}  ds \, dt.
\end{align*}
The term $\frac{|z - w|^2|s_2-t_1|^2}{|z-t_1|^2|s_2 - w|^2}$ is of bounded constant order. The remaining integral is upper bounded by $O(1/d^2)$ independently of $z$ and $w$. The $O(1/d^2)$ bound follows because the terms $P_{\tilde Q_d}(e_1 + (t-1)e_2, s e_1)\lesssim t^{-2}$ and $P_{\tilde Q_d}(s(e_1-1)+e_2, te_2)\lesssim s^{-2}$ and the rays $R_1, R_2$ start from $d$.

Hence, as $ d \to \infty $, the ratio $ \frac{L_d^2}{L_d^1} \to 0 $, proving the desired bound.
\end{proof}

\begin{proof}[Proof of Theorem \ref{thm:face_asymptotics}]
By Lemma~\ref{lem:face_four_diagram}, there are four possible base face diagrams of size $6$. By Prop \ref{prop:face_diagrams} each face diagram of size $6+2k$ is obtained by applying the face diagram extension operation $k$ times to one of these four base face diagram.  

Theorem~\ref{prop:ratio_decay} shows that a single extension reduces the probability of the occurrence of such a diagram at depth $d$ by a factor $O(d^{-2})$.  
Applying $k$ extensions multiplies the probability by $O(d^{-2k})$.  
For $k \ge 1$, this gives the claimed $O(d^{-2k})$ decay.  
When $k=0$, no decay occurs, so the relative probability of a size $6$ face tends to $1$ as $d \to \infty$.

There are four possible base diagrams for each size $6+2k$ Prop \ref{prop:face_diagrams}.  
Summing over the occurrences of all four diagrams only changes a constant factor, 
so the overall probability of a face of size $6+2k$ at depth $d$ is still $O(d^{-2k})$, as claimed.
\end{proof}

\bibliographystyle{alpha} 
\bibliography{refs} 

\end{document}